\documentclass[12pt,a4paper]{article}

\usepackage{amsmath,amssymb}	

\usepackage{graphicx}		  

\usepackage{tjm}
\usepackage[driver=dvips,text={372pt,533pt},centering]{geometry}

\usepackage{mathrsfs}
\usepackage{fancybox}
\usepackage{bm}
\usepackage{amsthm}
\usepackage[all]{xy} 
\usepackage{mathtools}
\usepackage{float}
\usepackage{hyperref}
\usepackage{longtable}
\usepackage{enumitem} 

\DeclareMathOperator{\Aut}{Aut}
\DeclareMathOperator{\Lie}{Lie}
\DeclareMathOperator{\tr}{tr}
\DeclareMathOperator{\ad}{ad}
\DeclareMathOperator{\Ad}{Ad}

\DeclareMathOperator{\id}{id}
\DeclareMathOperator{\Diff}{Diff}
\DeclareMathOperator{\Isom}{Isom}
\DeclareMathOperator{\GL}{GL}

\numberwithin{equation}{section}

\newtheorem{theorem}{Theorem}[section]
\newtheorem{lemma}[theorem]{Lemma}
\newtheorem{proposition}[theorem]{Proposition}
\newtheorem{corollary}[theorem]{Corollary}
\newtheorem{example}[theorem]{Example}
\newtheorem{definition}[theorem]{Definition}
\newtheorem{remark}[theorem]{Remark}
\title{Product structures and decompositions of $\mathbb{Z}_2\times\mathbb{Z}_2$-symmetric spaces}
\author{Jin Matsui}

\affils{Tokyo Metropolitan University}



\authorname{Jin Matsui}
\address{Department of Mathematical sciences\\
Tokyo Metropolitan University\\
Minami-Osawa, Hachioji, Tokyo, 192-0364 Japan}
\email{matsui-jin@ed.tmu.ac.jp}
\support{JST SPRING, Grant Number JPMJSP2156 and by MIYAKO-MIRAI Project of Tokyo Metropolitan University.}
\begin{document}
\maketitle
\begin{abstract}
In this paper, we establish the decomposition and product theory of $\Gamma$-symmetric spaces, where $\Gamma$ is isomorphic to $\mathbb{Z}_2\times\mathbb{Z}_2$.
These results lead to the local classification of semisimple Riemannian $\mathbb{Z}_2\times\mathbb{Z}_2$-symmetric spaces of nontrivial type.
Furthermore, we identify notable phenomena specific to $\mathbb{Z}_2\times\mathbb{Z}_2$-symmetric spaces: the existence of distinct product structures and that of an example with the trivial infinitesimal isotropy representation.
\end{abstract}
\section{Introduction}
Riemannian symmetric spaces were introduced by E. Cartan and have been extensively studied (cf. \cite{helgason1979differential}).
For instance, classification of Riemannian symmetric spaces is completed. 
The method consists of first decomposing the universal covering of a Riemannian symmetric space into irreducible factors and then characterizing each component by reducing the problem to the theory of Lie algebras. 
As a generalization of symmetric spaces, R. Lutz introduced $\Gamma$-symmetric spaces (cf. \cite{Lutz}). 
These spaces are, roughly speaking, manifolds equipped with a transformation group at each point that is isomorphic to a finite abelian group $\Gamma$.
Symmetric spaces are $\mathbb{Z}_2$-symmetric spaces in this sense, where $\mathbb{Z}_2$ denotes the cyclic group of order $2$.
In the case of $\mathbb{Z}_2\times\mathbb{Z}_2$-symmetric spaces, Bahturin and Goze (cf. \cite{Bahturin&Goze}), and Kollross (cf. \cite{Kollross}) classified them that can correspond to simple Lie algebras. 
This means that irreducible $\mathbb{Z}_2\times\mathbb{Z}_2$-symmetric spaces are locally classified (see lemma~\ref{types of irr.-Z2Z2}) and the next natural question is that ``Is it possible to decompose $\mathbb{Z}_2\times\mathbb{Z}_2$-symmetric spaces into irreducible ones?''.
This paper provides a partial answer to this question: the semisimple spaces of nontrivial type are decomposable (see definition~\ref{def-nontrivial type}, theorem~\ref{decomposition of semisimple Lie algebra of nontrival type} and theorem~\ref{decomposition of Gamma-symmetric spaces}).
This is a generalization of the case of symmetric spaces.
Moreover, we define the notion of product spaces of $\mathbb{Z}_2\times\mathbb{Z}_2$-symmetric spaces via the consideration of the decomposition.
Then we observe phenomena that the product structures are not necessarily unique (examples~\ref{example of product structures of a symmetric space and a Z2Z2-symmetric space}, \ref{example of product structures of two Z2Z2-symmetric space} and \ref{flag manifolds}) and the infinitesimal isotropy representation can be trivial (example~\ref{An example of a non-perpendicular decomposition}).
This situation does not arise when we consider only symmetric spaces, which means that the notion of products and decompositions of $\Gamma$-symmetric spaces is more than a mere analogy to those of ordinary ones. 

In sections~\ref{sec. Riemannian Gamma-symmetric spaces}, \ref{sec. Gamma-symmetric triples} and \ref{sec. Gamma-symmetric Lie algebra}, we review $\Gamma$-symmetric spaces, $\Gamma$-symmetric triples and $\Gamma$-symmetric Lie algebras, in which the latter two are concepts corresponding to $\Gamma$-symmetric spaces in the context of Lie groups and Lie algebras, respectively.
Note that our discussion is often specialized to the cases $\Gamma\cong\mathbb{Z}_2$ or $\mathbb{Z}_2\times\mathbb{Z}_2$, which suffice for our current purposes.
Furthermore, the decomposition of semisimple Riemannian $\mathbb{Z}_2\times\mathbb{Z}_2$-symmetric Lie algebras of nontrivial type is mentioned in section~\ref{sec. decomposition of nontrivial type}.
In section~\ref{sec. the relationship between Riemannian Gamma-symmetric triples, Riemannian Gamma-symmetric Lie algebras and Riemannian Gamma-symmetric spaces}, we formulate the relationship between the three concepts introduced above, which enables the local classification of semisimple Riemannian $\mathbb{Z}_2\times\mathbb{Z}_2$-symmetric spaces of nontrivial type.
The notable observations regarding product structures and the trivial infinitesimal isotropy representation are provided in sections~\ref{sec:Product structures} and \ref{sec. decomposition of nontrivial type}, respectively.
The discussion in this paper follows the arguments in \cite{ise1991lie}.
\section{Riemannian $\Gamma$-symmetric spaces}
\label{sec. Riemannian Gamma-symmetric spaces}
Let $M$ be a smooth manifold. 
First of all, we set
\begin{equation*}
    F(\phi,M)=\{y\in M\mid\phi(y)=y\},\;\;F(\Phi,M)=\{y\in M\mid\phi(y)=y\;(\forall\phi\in\Phi)\}
\end{equation*}
for a diffeomorphism $\phi$ of $M$ and a subgroup $\Phi$ of the diffeomorphism group $\Diff(M)$ of $M$. 
A connected $C^{\infty}$-manifold $M$ is called a \textit{symmetric space} if for each point $x\in M$ there exists an involution $s_x$ of $M$ with the following properties: $s_y\circ s_z=s_{s_y(z)}\circ s_y$ holds for all $y,z\in M$, and $x$ is an isolated point on $F(s_x,M)$. 
Here, a diffeomorphism $\phi :M\to M$ is an involution if $\phi$ satisfies $\phi\circ\phi=\id_M$. 
The involution $s_x$ is called the \textit{symmetry} at the point $x$.
Symmetric spaces are $\mathbb{Z}_2$-symmetric spaces in the sense of definition~\ref{Gamma symmetric space}. 
\begin{definition}[cf. \cite{Lutz}]\label{Gamma symmetric space}
     Let $M$ be a connected $C^{\infty}$-manifold and $\Gamma$ a finite abelian group. A \textit{$\Gamma$-symmetric structure} on $M$ is a family $\mu=\{\mu^{\gamma}\}_{\gamma\in\Gamma}$ of smooth maps $\mu^{\gamma}:M\times M\to M$ which satisfies the following conditions:
\begin{enumerate}
    \item[(S1)] For each $x\in M$, the map
    \begin{equation*}
        \phi_x:\Gamma\to\Diff(M); \gamma\mapsto \gamma_x:=\mu^{\gamma}(x,\cdot)
    \end{equation*}
        is an injective homomorphism.
    \item[(S2)] Each $x\in M$ is isolated in the fixed point set $F(\phi_x(\Gamma), M)$.
    \item[(S3)] For each $x\in M$ and $\gamma\in\Gamma$, $\gamma_x$ satisfies
    \begin{equation}\label{symmetries are automorphism}
        \gamma_x\circ\delta_y\circ\gamma_x^{-1}=\delta_{\gamma_x(y)}
    \end{equation}
for all $y\in M$ and for all $\delta\in\Gamma$.
\end{enumerate}
Then $(M,\Gamma,\mu)$ is called a \textit{$\Gamma$-symmetric space}.
\end{definition}
We will often write $(M,\Gamma)$ or $M$ for a $\Gamma$-symmetric space $(M,\Gamma,\mu)$.
We denote $\phi_x(\Gamma)$ by $\Gamma_x$ and call it the \textit{symmetric transformation group} at $x\in M$ and $\gamma_x\in\Gamma_x$ a \textit{symmetry} at $x\in M$. 
A Riemannian manifold $(M,g)$ is called a \textit{Riemannian $\Gamma$-symmetric space} if $M$ is a $\Gamma$-symmetric space and all symmetries at each point are isometries of $(M,g)$.
We write $(M,\Gamma,g)$ for the Riemannian $\Gamma$-symmetric space.
Now, let $(M,\Gamma,g)$ be a Riemannian $\Gamma$-symmetric space and $(M',\Gamma',g')$ a Riemannian $\Gamma'$-symmetric space.
\begin{definition}
    $(M,\Gamma,g)$ and $(M',\Gamma',g')$ are said to be \textit{isomorphic} if there exists a pair $(\varphi,\Phi)$ of an isometry $\varphi:M\to M'$ and a group isomorphism $\Phi:\Gamma\to\Gamma'$ such that 
    \begin{equation}\label{Definition of Gamma-isomorphism}
        \varphi\circ\gamma_x=\Phi(\gamma)_{\varphi(x)}\circ\varphi
    \end{equation}
    holds for each $x\in M$ and for each $\gamma\in\Gamma$.
    Then, we write $(M,\Gamma,g)\cong (M',\Gamma',g')$ and $(\varphi,\Phi)$ is called an \textit{isomorphism} from $(M,\Gamma,g)$ to $(M',\Gamma',g')$.
\end{definition}

Let $(\tilde{M},\tilde{g})$ be the Riemannian universal covering manifold of $(M,g)$. 
For $x\in M$ and $\tilde{x}\in \tilde{M}$ satisfying $\pi(\tilde{x})=x$, where $\pi$ denotes the natural projection
, let $\tilde{U}$ be a neighborhood of $\tilde{x}$ such that $\pi|_{\tilde{U}}:\tilde{U}\to\pi(\tilde{U})$ is an isometry.
Since $\Gamma$ is finite, $\tilde{V}:=\bigcap_{\gamma\in\Gamma}(\pi|_{\tilde{U}})^{-1}\left(\gamma_x(\pi(\tilde{U}))\right)\subset\tilde{U}$ is a neighborhood of $\tilde{x}$ satisfying $\Gamma_x(\pi(\tilde{V}))=\pi(\tilde{V})$.
Then, a local symmetry $\tilde{\gamma}_{\tilde{x}}$ at $\tilde{x}\in\tilde{M}$ can be defined by $\tilde{\gamma}_{\tilde{x}}:=(\pi|_{\tilde{V}})^{-1}\circ\gamma_x\circ\pi|_{\tilde{V}}$.
From \cite[corollary 6.4 Chapter VI]{Kobayashi-Nomizu}, $\tilde{\gamma}_{\tilde{x}}$ is uniquely extended to $\tilde{M}$.
Hence, $(\tilde{M},\Gamma,\tilde{g})$ is a Riemannian $\Gamma$-symmetric space.
We call $(\tilde{M},\Gamma,\tilde{g})$ the Riemannian universal covering $\Gamma$-symmetric space of $(M,\Gamma,g)$.

\begin{definition}
    $(M,\Gamma,g)$ and $(M',\Gamma',g')$ are said to be \textit{locally isomorphic} if the Riemannian universal covering $\Gamma$-symmetric space of $(M,\Gamma,g)$ and the Riemannian universal covering $\Gamma'$-symmetric space of $(M',\Gamma',g')$ are isomorphic. Then, we write $(M,\Gamma,g)\simeq (M',\Gamma',g')$.
\end{definition}

The automorphism group $\Aut(M,\Gamma,g)$ of $(M,\Gamma,g)$ is defined by 
\begin{equation}\label{definition of Aut}
    \Aut(M,\Gamma,g):=\{(\varphi,\Phi)\mid (\varphi,\Phi) \textrm{ is an isomorphism from $(M,\Gamma,g)$ to itself}\},
\end{equation}
which is a subgroup of $I(M,g)\times \Aut(\Gamma)$.
Here, $I(M,g)$ is the isometry group of $(M,g)$ and $\Aut(\Gamma)$ is the autmorphim group of $\Gamma$.
Since $\Aut(\Gamma)$ is finite, the topology of $\Aut(\Gamma)$ is discrete. 
Therefore, it is proved that $\Aut(M,\Gamma,g)$ is a closed subgroup of $I(M,g)\times \Aut(\Gamma)$. 
This means that $\Aut(M,\Gamma,g)$ is a Lie group.
The equation (\ref{symmetries are automorphism}) implies that $(\gamma_x,\id_{\Gamma})\in \Aut(M,\Gamma,g)$ for all $x\in M$ and for all $\gamma\in\Gamma$.
We define the action of $\Aut(M,\Gamma,g)$ on $M$ by $(\varphi,\Phi)\cdot x=\varphi(x)$ for $(\varphi,\Phi)\in\Aut(M,\Gamma,g)$ and $x\in M$.
Using exactly the same proof as for \cite[Theorem~1]{ledgerObata1968}, we have the following theorem:
\begin{theorem}[\cite{ledgerObata1968}]\label{LedgerObata}
    The automorphism group $\Aut(M,\Gamma,g)$ acts on $M$ transitively, that is, every Riemannian $\Gamma$-symmetric space is homogeneous. 
\end{theorem}
The identity component $\Aut_0(M,\Gamma,g)$ of $\Aut(M,\Gamma,g)$ is also a Lie group. 
Since the topology of $\Aut(\Gamma)$ is discrete, we can naturally regard $\Aut_0(M,\Gamma,g)$ as a subgroup of $I(M,g)$.
From theorem~\ref{LedgerObata} and the connectivity of $M$, we immediately have the next corollary:
\begin{corollary}\label{actstransitive}
    The connected Lie group $\Aut_0(M,\Gamma,g)$ acts on $(M,\Gamma,g)$ transitively. 
\end{corollary}

At the end of this section, we define the notion of subsymmetric structure:
\begin{definition}
    Let $\Sigma$ be a subgroup of $\Gamma$. 
    A $\Gamma$-symmetric space $(M,\Gamma,\mu)$ has a \textit{$\Sigma$-subsymmetric structure} $\mu^{\Sigma}=\{\mu^{\gamma}\}_{\gamma\in\Sigma}$ on $M$ if a subfamily $\mu^{\Sigma}$ of $\mu$ is a $\Sigma$-symmetric structure on $M$. 
    Furthermore, a $\Sigma$-subsymmetric structure is called \textit{trivial} if $\Sigma=\Gamma$. 
    A $\Gamma$-symmetric space $(M,\Gamma,\mu)$ is \textit{proper} if every $\Sigma$-subsymmetric structure of $\mu$ is trivial.
\end{definition}
If $\Gamma\cong\mathbb{Z}_2$, then a $\Gamma$-symmetric space $(M,\Gamma,\mu)$ is proper.
\section{Product structures of $\Gamma$-symmetric spaces}\label{sec:Product structures}
Let $\Gamma_i$, $i=1,2$, be finite abelian groups and $(M_i,\Gamma_i,\mu_i)$ $\Gamma_i$-symmetric spaces.
We write $\Gamma:=\Gamma_1\times\Gamma_2$ and $M:=M_1\times M_2$ and define a $\Gamma$-symmetric structure $\mu=\{\mu^{(\gamma_1,\gamma_2)}:M\times M\to M\}$, also denoted by $\mu_1\times\mu_2$, on $M$ by 
\begin{eqnarray*}
\mu^{(\gamma_1,\gamma_2)}((x_1,x_2),(y_1,y_2)):=(\mu_1^{\gamma_1}(x_1,y_1),\mu_2^{\gamma_2}(x_2,y_2))
\end{eqnarray*}
for all $(x_1,x_2),(y_1,y_2)\in M$ and for all $(\gamma_1,\gamma_2)\in\Gamma$.
Then, a $\Sigma$-subsymmetric structure $\mu^{\Sigma}$ of $\mu_1\times\mu_2$ is called a \textit{product structure} of $\mu_1$ and $\mu_2$ and $(M,\Sigma,\mu^{\Sigma})$ is called a \textit{product space} of $(M_1,\Gamma_1,\mu_1)$ and $(M_2,\Gamma_2,\mu_2)$.
Clearly, $\mu_1\times\mu_2$ is one of the product structures, however, $(M,\Gamma,\mu_1\times\mu_2)$ is generally not proper. 
Examples of proper product structures are given in examples~\ref{example of the product structure of two symmetric spaces}, \ref{example of product structures of a symmetric space and a Z2Z2-symmetric space} and \ref{example of product structures of two Z2Z2-symmetric space}.
In addition, let $(M_i,\Gamma_i,g_i)$, $i=1,2$, be Riemannian $\Gamma_i$-symmetric spaces and $(M_1\times M_2,\Sigma)$ a product space of $(M_1,\Gamma_1)$ and $(M_2,\Gamma_2)$. 
Then, all symmetries of $(M_1\times M_2,\Sigma)$ are isometries with respect to $g_1\oplus g_2$. 
We also call $(M_1\times M_2,\Sigma,g_1\oplus g_2)$ a \textit{product space} of $(M_1,\Gamma_1,g_1)$ and $(M_2,\Gamma_2,g_2)$.

\begin{example}\label{example of the product structure of two symmetric spaces}[The proper product space of two symmetric spaces]
When $M_1$ and $M_2$ are symmetric spaces, $M_1\times M_2$ has the product $\mathbb{Z}_2\times\mathbb{Z}_2$-symmetric structure $\mu_1\times\mu_2$. 
Then, $(M_1\times M_2,\mathbb{Z}_2\times\mathbb{Z}_2,\mu_1\times\mu_2)$ has the $\mathbb{Z}_2$-subsymmetric structure $\mu^{\mathbb{Z}_2}$ uniquely. 
$(M_1\times M_2,\mathbb{Z}_2,\mu^{\mathbb{Z}_2})$ has been known as the product symmetric space of the symmetric spaces $M_1$ and $M_2$.
\end{example}
\begin{example}\label{example of product structures of a symmetric space and a Z2Z2-symmetric space}[Proper product spaces of a symmetric space and a $\mathbb{Z}_2\times\mathbb{Z}_2$-symmetric space]
Suppose that $\Gamma_1\cong\mathbb{Z}_2$, $\Gamma_2\cong\mathbb{Z}_2\times\mathbb{Z}_2$ and $(M_2,\Gamma_2)$ is proper. 
We write $\Gamma_1=\langle\gamma_1\rangle$ and $\Gamma_2=\langle\gamma_2, \delta_2\rangle$ and then we have
\[
\Gamma=\langle(\gamma_1,\id_{M_2}),(\id_{M_1},\gamma_2),(\id_{M_1},\delta_2)\rangle\cong(\mathbb{Z}_2)^3.
\]
We will find a subgroup $\Sigma$ of $\Gamma$ such that
 $(M_1\times M_2,\Sigma,\mu^{\Sigma})$ is proper.
 If $\Sigma\cong\mathbb{Z}_2$, then $\mu^{\Sigma}$ does not satisfy (S2) since $(M_2,\Gamma_2)$ is proper. 
 Assume that $\Sigma\cong\mathbb{Z}_2\times\mathbb{Z}_2$.
We consider choosing two elements from $\Gamma$ generating $\Sigma$.
Since $M_2$ is proper, if $(M_1\times M_2,\Sigma,\mu^{\Sigma})$ satisfies (S2), then the two elements generating $\Sigma$ hold the following conditions: (a) For the $M_1$ components, at least one of the two elements contains $\gamma_1$; (b) For the $M_2$ components, the two elements do not contain $\id_{M_2}$; (c) The $M_2$ components of the two elements are mutually distinct.
There exist six elements satisfying (b) : $(\id_{M_1},\gamma_2)$, $(\id_{M_1},\delta_2)$, $(\id_{M_1},\gamma_2\delta_2)$, $(\gamma_1,\gamma_2)$, $(\gamma_1,\delta_2)$ and $(\gamma_1,\gamma_2\delta_2)$.
We will choose two of these six, denoted by $\alpha$ and $\beta$, such that they satisfy (a) and (c). 
$\alpha_{(i)}$ (resp. $\beta_{(i)}$) denotes the $M_i$ component of $\alpha$ (resp. $\beta$). 
From (a)-(c), if $\alpha_{(1)}=\id_{M_1}$ (resp. $\beta_{(1)}=\id_{M_1}$), then 
$\beta$ and $\alpha\beta$ (resp. $\alpha$ and $\alpha\beta$) include neither $\id_{M_1}$ nor $\id_{M_2}$.
Hence, the subgroup $\Sigma$ generated by $\alpha$ and $\beta$ contains generators that satisfy condition (d), namely, they include neither $\id_{M_1}$ nor $\id_{M_2}$.
The group $\Gamma$ has three elements satisfying (d): $(\gamma_1,\gamma_2)$, $(\gamma_1,\delta_2)$ and $(\gamma_1,\gamma_2\delta_2)$. 
Thus by choosing two with property (c) out of the three, we 
obtain three groups as candidates for $\Sigma$:
\begin{equation*}
    \Sigma_1=\langle(\gamma_1,\gamma_2),(\gamma_1,\delta_2)\rangle, \;\;
    \Sigma_2=\langle(\gamma_1,\gamma_2),(\gamma_1,\gamma_2\delta_2)\rangle 
    \;\;\mathrm{and}\;\;
    \Sigma_3=\langle(\gamma_1,\delta_2),(\gamma_1,\gamma_2\delta_2)\rangle.
\end{equation*}
We check $\Sigma_1$ satisfies (S2). 
Write $M=M_1\times M_2$ and take the origin $o=(o_1,o_2)\in M$. 
The differentials of ${\theta_i}_{o_i}\in{\Gamma_i}_{o_i}$ at $o_i$, denoted by ${\theta_i}_{*}$,  give the eigendecomposition of $T_{o_i}M_i$ and then they provide the following decomposition of $T_{o}M$:
\begin{equation}\label{example_of_decomposition_in_the_case_Z2_and_Z2Z2}
T_oM =T_{o_1}M_1^{-\gamma_1}\oplus \left(T_{o_2}M_2^{(+\gamma_2,-\delta_2)}\oplus T_{o_2}M_2^{(-\gamma_2,+\delta_2)}\oplus T_{o_2}M_2^{(-\gamma_2,-\delta_2)}\right),
\end{equation}
where 
\begin{eqnarray*}
    &&T_{o_1}M_1^{-\gamma_1}=\{v\in T_{o_1}M_1\mid {\gamma_1}_{*}(v)=-v\}, \\
    &&T_{o_2}M_2^{(+\gamma_2,-\delta_2)}=\{v\in T_{o_2}M_2\mid {\gamma_2}_{*}(v)=v, {\delta_2}_{*}(v)=-v\}, \\
    &&T_{o_2}M_2^{(-\gamma_2,+\delta_2)}=\{v\in T_{o_2}M_2\mid {\gamma_2}_{*}(v)=-v, {\delta_2}_{*}(v)=v\}, \\
    &&T_{o_2}M_2^{(-\gamma_2,-\delta_2)}=\{v\in T_{o_2}M_2\mid {\gamma_2}_{*}(v)=-v, {\delta_2}_{*}(v)=-v\} .
\end{eqnarray*}
Vectors of $T_oM$ fixed by both $(\gamma_1,\gamma_2)$ and $(\gamma_1,\delta_2)$ are only zero vector, which means that $\Sigma_1$ satisfies (S2).
Similarly, we can show that $\Sigma_2$ and $\Sigma_3$ also satisfy (S2).
Therefore, $M_1\times M_2$ has three $\Sigma_i$-symmetric structures $\mu^{\Sigma_i}$ and each product space $(M_1\times M_2,\Sigma_i,\mu^{\Sigma_i})$ is proper.
Finally, if $\Sigma\cong(\mathbb{Z}_2)^3$, that is, $\Sigma=\Gamma$, then $(M_1\times M_2,\Sigma,\mu_1\times\mu_2)$ is not proper.
\end{example}

\begin{example}\label{example of product structures of two Z2Z2-symmetric space}[Proper product spaces of two $\mathbb{Z}_2\times\mathbb{Z}_2$-symmetric spaces]
Suppose that $\Gamma_i\cong\mathbb{Z}_2\times\mathbb{Z}_2$ and $(M_i,\Gamma_i)$ is proper. 
Write $\Gamma_1=\langle\gamma_1,\delta_1\rangle$ and $\Gamma_2=\langle\gamma_2, \delta_2\rangle$ and then we have
\[
\Gamma=\langle(\gamma_1,\id_{M_2}),(\delta_1,\id_{M_2}),(\id_{M_1},\gamma_2),(\id_{M_1},\delta_2)\rangle\cong(\mathbb{Z}_2)^4.
\]
We will find $\Sigma$ such that $(M_1\times M_2,\Sigma)$ is proper.
Similarly to example~\ref{example of product structures of a symmetric space and a Z2Z2-symmetric space}, we first assume that $\Sigma\cong\mathbb{Z}_2\times\mathbb{Z}_2$.
We choose two elements of $\Gamma$ generating $\Sigma$.
Since $(M_1,\Gamma_1)$ and $(M_2,\Gamma_2)$ are proper, if $(M_1\times M_2,\Sigma)$ satisfies (S2), then the two hold the following conditions: (c$'$) For $i=1,2$, the $M_i$ components of the two are mutually distinct; (d$'$) The two elements contain neither $\id_{M_1}$ nor $\id_{M_2}$. 
The group $\Gamma$ has nine elements with property (d$'$): 
$(\gamma_1,\gamma_2)$, $(\gamma_1,\delta_2)$, $(\gamma_1,\gamma_2\delta_2)$, $(\delta_1,\gamma_2)$, $(\delta_1,\delta_2)$, $(\delta_1,\gamma_2\delta_2)$, $(\gamma_1\delta_1,\gamma_2)$, $(\gamma_1\delta_1,\delta_2)$ and $(\gamma_1\delta_1,\gamma_2\delta_2)$. 
For each of them, we can choose four elements satisfying (c$'$). 
Since the order of selection does not matter, the number of the pairs is eighteen. 
On the other hand, for each $\Sigma$ the number of such pairs generating $\Sigma$ is three.
Hence we have six groups as candidates for $\Sigma$, up to the choice of the generators:
\begin{equation}\label{6 group Sigma}
\begin{array}{ll}
    \Sigma_1:=\langle(\gamma_1,\gamma_2),(\delta_1,\delta_2)\rangle, 
    &\Sigma_2:=\langle(\gamma_1,\gamma_2),(\delta_1,\gamma_2\delta_2)\rangle, \\
    \Sigma_3:=\langle(\gamma_1,\delta_2),(\delta_1,\gamma_2)\rangle,
    &\Sigma_4:=\langle(\gamma_1,\delta_2),(\delta_1,\gamma_2\delta_2)\rangle, \\
    \Sigma_5:=\langle(\gamma_1,\gamma_2\delta_2),(\delta_1,\gamma_2)\rangle,
    &\Sigma_6:=\langle(\gamma_1,\gamma_2\delta_2),(\delta_1,\delta_2)\rangle.
\end{array}
\end{equation}
Conversely, they satisfy (S2), that is, $\mu^{\Sigma_i}$, $i=1,\ldots,6$, are $\Sigma_i$-symmetric structures on $M$.
As we see later, $\mu^{\Sigma_i}$ are generally not isomorphic. 
Next, assume that $\Sigma\cong(\mathbb{Z}_2)^3$.
If $\mu^{\Sigma}$ satisfies (S2), then there exist two elements of $\Sigma$ satisfying (c$'$) and (d$'$).
This fact is proved as follows:
Let $\alpha$, $\beta$ and $\zeta$ be generators of $\Sigma$.
We can assume that $\alpha_{(1)}\neq\id_{M_1}$, $\beta_{(1)}\neq\id_{M_1}$ and $\alpha_{(1)}\neq\beta_{(1)}$.
Now, we need to consider the three cases:  
\begin{enumerate}
    \item $\alpha_{(2)}=\beta_{(2)}$, 
    \item $\alpha_{(2)}\neq\beta_{(2)}$, $\alpha_{(2)}\neq\id_{M_2}$ and $\beta_{(2)}=\id_{M_2}$,
    \item $\alpha_{(2)}\neq\beta_{(2)}$, $\alpha_{(2)}\neq\id_{M_2}$ and $\beta_{(2)}\neq\id_{M_2}$. 
\end{enumerate}
In the first case, we obtain $\alpha_{(2)}=\beta_{(2)}\neq\id_{M_2}$, $\zeta_{(2)}\neq\id_{M_2}$ and $\alpha_{(2)}\neq\zeta_{(2)}$, and then the pairs $\{\alpha,\zeta\}$, $\{\beta,\zeta\}$ or $\{\alpha,\alpha\beta\zeta\}$ satisfy (c$'$) and (d$'$).
In the second case, we obtain $\zeta_{(2)}\neq\id_{M_2}$ and $\alpha_{(2)}\neq\zeta_{(2)}$, and then the pairs $\{\alpha,\zeta\}$ or $\{\alpha,\beta\zeta\}$ satisfy (c$'$) and (d$'$).
In the third case, the pair $\{\alpha,\beta\}$ satisfies them.
Hence, $\Sigma$ has two elements satisfying (c$'$) and (d$'$), that is, it has at least one of the subgroups $\Sigma_1$, $\Sigma_2$, $\ldots$, $\Sigma_5$ or $\Sigma_6$ in (\ref{6 group Sigma}), which means that $(M,\Sigma)$ is not proper.
Finally, if $\Sigma\cong(\mathbb{Z}_2)^4$, that is, $\Sigma=\Gamma$, then $(M,\Sigma,\mu_1\times\mu_2)$ is not proper.
\end{example}

We prepare to provide examples that $\mu^{\Sigma_1}$, $\mu^{\Sigma_2}$, $\ldots$, $\mu^{\Sigma_6}$ are pairwise non-isomorphic.
Let $(M_i,\Gamma_i,g_i)$ ($\Gamma_i\cong\mathbb{Z}_2\times\mathbb{Z}_2$, $i=1,2$) be Riemannian $\Gamma_i$-symmetric spaces.
We use the notation from examples~\ref{example of product structures of a symmetric space and a Z2Z2-symmetric space} and \ref{example of product structures of two Z2Z2-symmetric space}.
Similarly to the decomposition (\ref{example_of_decomposition_in_the_case_Z2_and_Z2Z2}), 
we can decompose $T_{o}M$:
\begin{equation*}
\begin{split}
    T_oM =&(T_{o_1}M_1^{(+\gamma_1,-\delta_1)}\oplus T_{o_1}M_1^{(-\gamma_1,+\delta_1)}\oplus T_{o_1}M_1^{(-\gamma_1,-\delta_1)}) \\
    &\oplus (T_{o_2}M_2^{(+\gamma_2,-\delta_2)}\oplus T_{o_2}M_2^{(-\gamma_2,+\delta_2)}\oplus T_{o_2}M_2^{(-\gamma_2,-\delta_2)}).
\end{split}
\end{equation*}
Put
\begin{equation*}
    \begin{split}
        &d^{+-}_1=\dim T_{o_1}M_1^{(+\gamma_1,-\delta_1)}, \quad
        d^{-+}_1=\dim T_{o_1}M_1^{(-\gamma_1,+\delta_1)}, \quad
        d^{--}_1=\dim T_{o_1}M_1^{(-\gamma_1,-\delta_1)}, \\
        &d^{+-}_2=\dim T_{o_2}M_2^{(+\gamma_2,-\delta_2)}, \quad
        d^{-+}_2=\dim T_{o_2}M_2^{(-\gamma_2,+\delta_2)}, \quad
        d^{--}_2=\dim T_{o_2}M_2^{(-\gamma_2,-\delta_2)}.
    \end{split}
\end{equation*}
Suppose that $d^{+-}_1, d^{-+}_1, d^{--}_1$ are pairwise distinct, and that $d^{+-}_2, d^{-+}_2, d^{--}_2$ are also pairwise distinct.
We set
\begin{equation*}
\begin{split}
    &D_1 = \{d^{+-}_1, d^{-+}_1, d^{--}_1\}, \quad
    D_2 = \{d^{+-}_2, d^{-+}_2, d^{--}_2\}, \quad
    D = D_1 \times D_2.
\end{split}
\end{equation*}
We focus on $\Sigma_1$ and $\Sigma_2$.
If $(M,\Sigma_1,g)$ and $(M,\Sigma_2,g)$ are isomorphic, where $g=g_1\oplus g_2$, then there exist $\varphi\in I(M,g)$ and $\Phi\in \Isom(\Sigma_1,\Sigma_2)$ such that $\varphi\circ\gamma_x=\Phi(\gamma)_{\varphi(x)}\circ\varphi$ holds for all $\gamma\in\Sigma_1$ and for all $x\in M$, where
\begin{eqnarray*}
\Isom(\Sigma_1,\Sigma_2) 
&:=& \{\Phi:\Sigma_1\to\Sigma_2;\mathrm{a\;group\;isomorphism}\} \\
&=&
\left\{
\begin{aligned}
&\Phi_1: (\gamma_1,\gamma_2)\mapsto(\gamma_1,\gamma_2),\;
      (\delta_1,\delta_2)\mapsto(\delta_1,\gamma_2\delta_2), \\
&\Phi_2: (\gamma_1,\gamma_2)\mapsto(\gamma_1,\gamma_2),\;
      (\delta_1,\delta_2)\mapsto(\gamma_1\delta_1,\delta_2), \\
&\Phi_3: (\gamma_1,\gamma_2)\mapsto(\delta_1,\gamma_2\delta_2),\;
      (\delta_1,\delta_2)\mapsto(\gamma_1,\gamma_2), \\
&\Phi_4: (\gamma_1,\gamma_2)\mapsto(\delta_1,\gamma_2\delta_2),\;
      (\delta_1,\delta_2)\mapsto(\gamma_1\delta_1,\delta_2), \\
&\Phi_5: (\gamma_1,\gamma_2)\mapsto(\gamma_1\delta_1,\delta_2),\;
      (\delta_1,\delta_2)\mapsto(\gamma_1,\gamma_2), \\
&\Phi_6: (\gamma_1,\gamma_2)\mapsto(\gamma_1\delta_1,\delta_2),\;
      (\delta_1,\delta_2)\mapsto(\delta_1,\gamma_2\delta_2)
\end{aligned}
\right\}.
\end{eqnarray*}
Since $(M,\Sigma_2,g)$ is homogeneous, we can assume that $\varphi(o)=o$. 
 Then, $\varphi$ preserves the eigendecompositions of $T_oM$.
For example, when $\Phi=\Phi_1$, we obtain 
\begin{equation}\label{used when give example not isomorphic}
  \begin{split}
    &\varphi_*\big(T_{o_1}M_1^{(+\gamma_1,-\delta_1)}\oplus T_{o_2}M_2^{(+\gamma_2,-\delta_2)}\big)=T_{o_1}M_1^{(+\gamma_1,-\delta_1)}\oplus T_{o_2}M_2^{(+\gamma_2,-\delta_2)}, \\
    &\varphi_*\big(T_{o_1}M_1^{(-\gamma_1,+\delta_1)}\oplus T_{o_2}M_2^{(-\gamma_2,+\delta_2)}\big)=T_{o_1}M_1^{(-\gamma_1,+\delta_1)}\oplus T_{o_2}M_2^{(-\gamma_2,-\delta_2)}, \\
   &\varphi_*\big(T_{o_1}M_1^{(-\gamma_1,-\delta_1)}\oplus T_{o_2}M_2^{(-\gamma_2,-\delta_2)}\big)=T_{o_1}M_1^{(-\gamma_1,-\delta_1)}\oplus T_{o_2}M_2^{(-\gamma_2,+\delta_2)}.
  \end{split}   
\end{equation}
Based on \eqref{used when give example not isomorphic}, set
\begin{equation*}
\begin{split}
    &S_1=\{(d^{+-}_1,d^{+-}_2),(d^{-+}_1,d^{-+}_2),(d^{--}_1,d^{--}_2)\}, \\
    &S_2=\{(d^{+-}_1,d^{+-}_2),(d^{-+}_1,d^{--}_2),(d^{--}_1,d^{-+}_2)\}
\end{split}
\end{equation*}
and define a bijection $\psi:S_1\to S_2$ by $\psi((d^{+-}_1,d^{+-}_2))=(d^{+-}_1,d^{+-}_2)$, 
$\psi((d^{-+}_1,d^{-+}_2))=(d^{-+}_1,d^{--}_2)$, 
$\psi((d^{--}_1,d^{--}_2))=(d^{--}_1,d^{-+}_2)$.
Then, every $u=(u_1,u_2)\in S_1$ and $\psi(u)\in S_2$ satisfy 
\begin{equation}\label{dim_situ}
    u_1+u_2=\text{Pr}_1\big(\psi(u)\big)+\text{Pr}_2\big(\psi(u)\big),
\end{equation}
where $\text{Pr}_1:D\to D_1$ and $\text{Pr}_2:D\to D_2$ are the natural projections. 
Lemma~\ref{dim} below shows that certain assumptions on the dimensions of the simultaneous eigenspaces contradict \eqref{dim_situ}.
\begin{lemma}\label{dim}
Let $d^{+-}_1, d^{-+}_1, d^{--}_1, d^{+-}_2, d^{-+}_2, d^{--}_2$ be nonnegative integers such that $d^{+-}_1, d^{-+}_1, d^{--}_1$ are pairwise distinct and that $d^{+-}_2, d^{-+}_2, d^{--}_2$ are also pairwise distinct.
Let $D_1 = \{d^{+-}_1, d^{-+}_1, d^{--}_1\}$, $D_2 = \{d^{+-}_2, d^{-+}_2, d^{--}_2\}$ and $D = D_1 \times D_2$.
Let $\text{Pr}_1:D\to D_1$ and $\text{Pr}_2:D\to D_2$ be the natural projections. 
Take arbitrary subsets $S_1$ and $S_2$ of $D$ each consisting of exactly three pairs, satisfying $\text{Pr}_1(S_1)=\text{Pr}_1(S_2)=D_1$, $\text{Pr}_2(S_1)=\text{Pr}_2(S_2)=D_2$ and $S_1 \neq S_2$.
Suppose that the following conditions hold simultaneously:
\begin{enumerate}
    \item $d^{+-}_1 < d^{-+}_1 < d^{--}_1 < d^{+-}_2 < d^{-+}_2 < d^{--}_2$, \label{<}
    \item $d^{--}_2 - d^{-+}_2 > d^{--}_1 - d^{+-}_1$. \label{d2>}
\end{enumerate}
Then, for any bijection $\psi$ from $S_1$ to $S_2$, there exists some $u=(u_1, u_2) \in S_1$ such that
\[
    u_1 + u_2 \neq \text{Pr}_1(\psi(u))+\text{Pr}_2(\psi(u)).
\]
\end{lemma}
\begin{proof}
We proceed by contradiction. Suppose that there exists a bijection $\psi:S_1\to S_2$ such that $u_1 + u_2 = \text{Pr}_1(\psi(u))+\text{Pr}_2(\psi(u))$ holds for all $u=(u_1, u_2) \in S_1$.
Take $u^*\in S_1$ with $\text{Pr}_2(u^*)=d^{--}_2$ and set $u^*_1=\text{Pr}_1(u^*)$.
By the assumption for contradiction, we have $u^*_1+d^{--}_2=\text{Pr}_1(\psi(u^*))+\text{Pr}_2(\psi(u^*))$. 
Thus, from the condition~\ref{d2>}, we obtain
\[
    \text{Pr}_1(\psi(u^*))+\text{Pr}_2(\psi(u^*))=u^*_1+d^{--}_2\geq d^{+-}_1 + d^{--}_2 > d^{--}_1 + d^{-+}_2
\]
because $u^*_1\geq d^{+-}_1$.
On the other hand, every $u=(u_1, u_2) \in D$ with $u_2\neq d^{--}_2$ satisfies $u_1+u_2\leq d^{--}_1 + d^{-+}_2$ because $u_1\leq d^{--}_1$ and $u_2\leq d^{-+}_2$.
Therefore, it follows that $\text{Pr}_2(\psi(u^*))=d^{--}_2$, which yields $u^*_1=\text{Pr}_1(\psi(u^*))$.
Hence, we have $(u^*_1, d^{--}_2)\in S_1\cap S_2$.
Next, define integers $a, b \ (a < b)$ such that $\{a, b\} = D_1 \setminus \{u^*_1\}$.
Then, $S_1\setminus\{(u^*_1,d^{--}_2)\}$ coincides with either $\{(a,d^{+-}_2),(b,d^{-+}_2)\}$ or $\{(a,d^{-+}_2),(b,d^{+-}_2)\}$.
Similarly, $S_2\setminus\{(u^*_1,d^{--}_2)\}$ is either of them.
Thus, noting that $S_1\neq S_2$, if $S_1\setminus\{(u^*_1,d^{--}_2)\}=\{(a,d^{+-}_2),(b,d^{-+}_2)\}$, then $\psi((a, d^{+-}_2))$ is either $(a, d^{-+}_2)$ or $(b, d^{+-}_2)$.
From the assumption for contradiction, it follows that $d^{+-}_2 = d^{-+}_2$ or $a = b$, which contradicts the condition~\ref{<}.
Similarly, if $S_1\setminus\{(u^*_1,d^{--}_2)\}=\{(a,d^{-+}_2),(b,d^{+-}_2)\}$, then $\psi((a,d^{-+}_2))$ is either $(a, d^{+-}_2)$ or $(b, d^{-+}_2)$, which implies $d^{-+}_2 = d^{+-}_2$ or $a = b$.
This contradicts the condition~\ref{<}, which completes the proof.
\end{proof}
If the conditions~\ref{<} and \ref{d2>} in lemma~\ref{dim} hold, then $\Phi\neq\Phi_1$.
Similarly, it follows that $\Phi$ is not equal to any of $\Phi_2, \ldots, \Phi_6$.
Hence, $(M,\Sigma_1,g)$ and $(M,\Sigma_2,g)$ are not isomorphic.
By the same argument, the conditions~\ref{<} and \ref{d2>} in lemma~\ref{dim} show that $(M_1\times M_2,\Sigma_i,g)$ and $(M_1\times M_2,\Sigma_j,g)$ are not isomorphic for all $1\leq i, j \leq 6$ with $i\neq j$.

\begin{example}[Flag manifolds]\label{flag manifolds}
Let $r_1$, $r_2$, $r_3$, $r_4$ and $l$ be nonnegative integers with $r_1 + r_2 + r_3 + r_4 = l$ and $\mathcal{M}(p,q)$ the set of all $p\times q$ real matrices.
Define $\mathfrak{g}:=\mathfrak{so}(l)$.
All elements $X$ of $\mathfrak{so}(l)$ are written as
\[
X=
\left(
\begin{array}{c|c|c|c}
X_1 & A_1 & B_1 & C_1 \\
\hline
-{}^{t}\!A_1 & X_2 & C_2 & B_2 \\
\hline
-{}^{t}\!B_1 & -{}^{t}C_2 & X_3 & A_2 \\
\hline
-{}^{t}C_1 & -{}^{t}\!B_2 & -{}^{t}\!A_2 & X_4
\end{array}
\,\right)
, \mathrm{where}
\begin{array}{l}
A_1 \in\mathcal{M}(r_1,r_2), B_1 \in\mathcal{M}(r_1,r_3), \\
C_1 \in\mathcal{M}(r_1,r_4), C_2 \in\mathcal{M}(r_2,r_3), \\ 
B_2 \in\mathcal{M}(r_2,r_4), A_2 \in\mathcal{M}(r_3,r_4) \\
\mathrm{and}\; X_i \in \mathfrak{so}(r_i), i=1,\dots,4. 
\end{array}
\]
Consider the following automorphisms $\sigma$, $\tau\in\Aut(\mathfrak{g})$:
\begin{eqnarray*}
    \sigma(X):=S_1XS_1^{-1}, \quad \tau(X):=S_2XS_2^{-1},
\end{eqnarray*}
where 
\[
S_1:=
\left(
\begin{array}{cccc}
I_{r_1} & 0 & 0 & 0 \\
0 & I_{r_2} & 0 & 0 \\
0 & 0 & -I_{r_3} & 0 \\
0 & 0 & 0 & -I_{r_4}
\end{array}
\right),
\quad
S_2:=
\left(
\begin{array}{cccc}
I_{r_1} & 0 & 0 & 0 \\
0 & -I_{r_2} & 0 & 0 \\
0 & 0 & I_{r_3} & 0 \\
0 & 0 & 0 & -I_{r_4}
\end{array}
\right)
\]
and $I_{r_i}$ denotes the identity matrix of size $r_i$.
Then we obtain $\langle\sigma,\tau\rangle\cong\mathbb{Z}_2 \times \mathbb{Z}_2$ and $\sigma$, $\tau$ give the eigendecomposition of $\mathfrak{so}(l)$:
\[
\mathfrak{g}=\mathfrak{g}^{(+1,+1)}\oplus
\mathfrak{g}^{(+1,-1)}\oplus
\mathfrak{g}^{(-1,+1)}\oplus
\mathfrak{g}^{(-1,-1)},
\]
where 
for $\theta\in\Aut(\mathfrak{g})$ the eigenspaces $\mathfrak{g}^{\theta}$ and $\mathfrak{g}^{-\theta}$ are given by
\begin{equation}\label{eigenspace} 
    \mathfrak{g}^{\theta}=\{X\in\mathfrak{g}\mid \theta X=X\}, \quad \mathfrak{g}^{-\theta}=\{X\in\mathfrak{g}\mid \theta X=-X\} 
\end{equation}
and the simultaneous eigenspaces are defined by
\begin{eqnarray}\label{simultaneous eigenspace} 
\begin{aligned}
&\mathfrak{g}^{(+1,+1)}=\mathfrak{g}^{\sigma}\cap\mathfrak{g}^{\tau}, \;\;
\mathfrak{g}^{(+1,-1)}=\mathfrak{g}^{\sigma}\cap\mathfrak{g}^{-\tau}, \\
&\mathfrak{g}^{(-1,+1)}=\mathfrak{g}^{-\sigma}\cap\mathfrak{g}^{\tau}
\textrm{\;\;and\;\;}
\mathfrak{g}^{(-1,-1)}=\mathfrak{g}^{-\sigma}\cap\mathfrak{g}^{-\tau}. 
\end{aligned}
\end{eqnarray}
All elements of each simultaneous eigenspace $\mathfrak{g}^{(+1,+1)}$, $
\mathfrak{g}^{(+1,-1)}$, $
\mathfrak{g}^{(-1,+1)}$ and $\mathfrak{g}^{(-1,-1)}$ are expressed as 
\[
\left(
\begin{array}{cccc}
X_1 & 0 & 0 & 0 \\
0 & X_2 & 0 & 0 \\
0 & 0 & X_3 & 0 \\
0 & 0 & 0 & X_4
\end{array}
\right),
\qquad
\left(
\begin{array}{cccc}
0 & A_1 & 0 & 0 \\
-{}^{t}\!A_1 & 0 & 0 & 0 \\
0 & 0 & 0 & A_2 \\
0 & 0 & -{}^{t}\!A_2 & 0
\end{array}
\right),
\]
\[
\left(
\begin{array}{cccc}
0 & 0 & B_1 & 0 \\
0 & 0 & 0 & B_2 \\
-{}^{t}\!B_1 & 0 & 0 & 0 \\
0 & -{}^{t}\!B_2 & 0 & 0
\end{array}
\right)
\;\mathrm{and}\;
\left(
\begin{array}{cccc}
0 & 0 & 0 & C_1 \\
0 & 0 & C_2 & 0 \\
0 & -{}^{t}C_2 & 0 & 0 \\
-{}^{t}C_1 & 0 & 0 & 0
\end{array}
\right),
\]
respectively.
It is well-known that the eigendecomposition of $\mathfrak{so}(l)$, which has $\mathbb{Z}_2 \times \mathbb{Z}_2$-graded structure, defines the $\mathbb{Z}_2 \times \mathbb{Z}_2$-symmetric structure on the real flag manifold $F:=SO(l)/S(O(r_1)\times O(r_2)\times O(r_3)\times O(r_4))$ and the Killing form on $\mathfrak{so}(l)$ induces a metric on $(F,\mathbb{Z}_2\times\mathbb{Z}_2)$.
Note that $\dim(\mathfrak{g}^{(+1,-1)})=r_1r_2+r_3r_4$, 
$\dim(\mathfrak{g}^{(-1,+1)})=r_1r_3+r_2r_4$ and 
$\dim(\mathfrak{g}^{(-1,-1)})=r_1r_4+r_2r_3$.
If $r_1r_2\neq0$ and $r_3=r_4=0$, then $F$ is a symmetric space.
If $r_1r_2r_3\neq0$, then $F$ is a $\mathbb{Z}_2\times\mathbb{Z}_2$-symmetric space.
Now, we focus on the two flag manifolds with $(r_1,r_2,r_3,r_4)=(1,2,4,0)$ and $(r_1,r_2,r_3,r_4)=(3,2,1,r)$, i.e., define $M_1:=SO(7)/S(O(1)\times O(2)\times O(4))$, $M_2:=SO(11)/S(O(3)\times O(2)\times O(1)\times O(r))$. 
Then we get $d^{+-}_1=\dim(\mathfrak{g}^{(+1,-1)}_1)=2$, 
$d^{-+}_1=\dim(\mathfrak{g}^{(-1,+1)}_1)=4$, 
$d^{--}_1=\dim(\mathfrak{g}^{(-1,-1)}_1)=8$, $d^{+-}_2=\dim(\mathfrak{g}^{(+1,-1)}_2)=r+6$, 
$d^{-+}_2=\dim(\mathfrak{g}^{(-1,+1)}_2)=2r+3$ and 
$d^{--}_2=\dim(\mathfrak{g}^{(-1,-1)}_2)=3r+2$.
Hence, if $r>7$, then the assumptions for lemma~\ref{dim} hold.
Therefore, we obtain an infinite number of examples that $(M_1\times M_2,\Sigma_i,g)$ and $(M_1\times M_2,\Sigma_j,g)$ $(1\leq i\neq j\leq 6)$ are not isomorphic. 
\end{example}

A $\Gamma$-symmetric structure is generally independent of the generators of $\Gamma$. 
However, as seen in the above examples, product structures are sensitive to the choice of generators for $\Gamma_i$ and not necessarily unique even if they are proper.
This is a phenomenon specific to $\Gamma$-symmetric spaces, which does not arise when we consider only symmetric spaces.

\vspace{0.2cm}
Next, we introduce an expression of product spaces using generators of $\Gamma$.
For a given connected smooth manifold $M$, let $\gamma$ and $\delta$ satisfy the following conditions:
\begin{itemize}
\item $\Gamma=\langle\gamma,\delta\rangle\cong\mathbb{Z}_2$ or $\mathbb{Z}_2\times\mathbb{Z}_2$.
\item $(M,\langle\gamma,\delta\rangle)$ is a $\Gamma$-symmetric space.
\end{itemize}
Then $(M_1\times M_2,\langle(\gamma_1,\gamma_2),(\delta_1,\delta_2)\rangle)$ is a product space of a $\Gamma_1$-symmetric space $(M_1,\langle\gamma_1,\delta_1\rangle)$ and a $\Gamma_2$-symmetric space $(M_2,\langle\gamma_2,\delta_2\rangle)$.
We will write $(M_1,\langle\gamma_1,\delta_1\rangle)\times(M_2,\langle\gamma_2,\delta_2\rangle)$ as $\left(M_1\times M_2,\langle(\gamma_1,\gamma_2),(\delta_1,\delta_2)\rangle\right)$. 
As we have already seen in examples~\ref{example of the product structure of two symmetric spaces}, \ref{example of product structures of a symmetric space and a Z2Z2-symmetric space} and \ref{example of product structures of two Z2Z2-symmetric space}, every proper product space can be expressed in this notation by changing generators.
For example, $(M_1\times M_2,\Sigma_1)$ in (\ref{6 group Sigma}) is expressed as $(M_1,\langle\gamma_1,\delta_1\rangle)\times(M_2,\langle\gamma_2,\delta_2\rangle)$. 
Note that this expression is not unique.
It is easy to check the following conditions:
\begin{eqnarray*}
&&(M_1,\langle\gamma_1,\delta_1\rangle)\times(M_2,\langle\gamma_2,\delta_2\rangle) 
\cong(M_2,\langle\gamma_2,\delta_2\rangle)\times(M_1,\langle\gamma_1,\delta_1\rangle),\\
&&
\begin{aligned}    ((M_1,\langle\gamma_1,\delta_1\rangle)&\times(M_2,\langle\gamma_2,\delta_2\rangle))\times(M_3,\langle\gamma_3,\delta_3\rangle) \\
&=(M_1,\langle\gamma_1,\delta_1\rangle)\times((M_2,\langle\gamma_2,\delta_2\rangle)\times(M_3,\langle\gamma_3,\delta_3\rangle)).
\end{aligned}
\end{eqnarray*}

\vspace{0.2cm}
Finally, a Rimannian $\Gamma$-symmetric space $(M,\Gamma,g)$ is said to be \textit{locally reducible} if there exist positive-dimensional Rimannian $\Gamma_i$-symmetric spaces $(M_i,\Gamma_i,g_i)$, $i=1,2$, and a direct product $(M_1,\Gamma_1,g_1)\times (M_2,\Gamma_2,g_2)$ such that $(M,\Gamma,g)\simeq (M_1,\Gamma_1,g_1)\times (M_2,\Gamma_2,g_2)$. 
A Rimannian $\Gamma$-symmetric space is called \textit{locally irreducible} if it is not locally reducible.
\section{$\Gamma$-symmetric triples}\label{sec. Gamma-symmetric triples}
\subsection{$\Gamma$-symmetric triples}
Let $\Gamma$ be a finite abelian group again. We denote the automorphism group of a Lie group $G$ by $\Aut(G)$. Let $\rho:\Gamma\to\Aut(G)$ be an injective homomorphism. To shorten notation, we also write $\Gamma$ instead of $\rho(\Gamma)$. 
\begin{definition}[cf. \cite{Bahturin&Goze}]
Let $G$ be a connected Lie group and $K$ a closed subgroup of $G$ satisfying 
    \begin{equation*}
        F_0(\Gamma,G)\subset K\subset F(\Gamma,G),
    \end{equation*}
    where $F_0(\Gamma,G)$ denotes the connected component of $F(\Gamma,G)$ containing the identity element of $G$. 
    Then $(G,K,\Gamma)$ is called a \textit{$\Gamma$-symmetric triple}. 
\end{definition}
\begin{definition}\label{def of isom of triples}
     Let $(G,K,\Gamma)$ be a $\Gamma$-symmetric triple and $(G',K',\Gamma')$ a $\Gamma'$-symmetric triple.
     $(G,K,\Gamma)$ and $(G',K',\Gamma')$ are \textit{isomorphic} if there exist a Lie group isomorphism $\varphi : G\to G'$ and a group isomorphism $\Phi : \Gamma \to \Gamma'$ such that
     \begin{equation*}
         \mathrm{(A)}\quad\varphi\circ\gamma=\Phi(\gamma)\circ\varphi \textrm{\; for all }\gamma\in\Gamma
         \qquad\textrm{ and }\qquad 
         \mathrm{(B)}\quad\varphi(K)=K'.
     \end{equation*}
Then we write $(G,K,\Gamma)\cong(G',K',\Gamma')$ and call the pair $(\varphi,\Phi)$ an \textit{isomorphism}.
\end{definition}
We should strictly write not $\Phi : \Gamma \to \Gamma'$ but $\Phi : \rho_G(\Gamma) \to \rho_{G'}(\Gamma')$.
The abbreviated notation causes no confusion by the injectivity of $\rho_G$ and $\rho_{G'}$.

\begin{definition}
     Let $\Sigma$ be a subgroup of $\Gamma$. A triplet $(G,K,\Sigma)$ is a \textit{$\Sigma$-subsymmetric triple} of a $\Gamma$-symmetric triple $(G,K,\Gamma)$ if $(G,K,\Sigma)$ satisfies $F_0(\Sigma,G)\subset K\subset F(\Sigma,G)$. 
\end{definition}

Let $(G_i,K_i,\Gamma_i)$, $i=1,2$, be $\Gamma_i$-symmetric triples. 
We write $\Gamma=\Gamma_1\times\Gamma_2$, $G=G_1\times G_2$ and $K=K_1\times K_2$.
Let $(G,K,\Sigma)$ be a $\Sigma$-subsymmetric triple of $(G,K,\Gamma)$ with a subgroup $\Sigma$ of $\Gamma$.
Then, $(G,K,\Sigma)$ is called a \textit{direct product} of $(G_1,K_1,\Gamma_1)$ and $(G_2,K_2,\Gamma_2)$.

Let $\sigma$ and $\tau$ be involutions of $G$ satisfying $\sigma\tau=\tau\sigma$ and $F_0(\langle\sigma,\tau\rangle,G)\subset K\subset F(\langle\sigma,\tau\rangle,G)$.
If $\langle\sigma,\tau\rangle\cong\mathbb{Z}_2\times\mathbb{Z}_2$, then $(G,K,\langle\sigma,\tau\rangle)$ is a $\mathbb{Z}_2\times\mathbb{Z}_2$-symmetric triple. 
If $\langle\sigma,\tau\rangle\cong\mathbb{Z}_2$, then $(G,K,\langle\sigma,\tau\rangle)$ is a $\mathbb{Z}_2$-symmetric triple.
A $\mathbb{Z}_2$-symmetric triple $(G,K,\mathbb{Z}_2)$ is also written as $(G,K,\theta)$ and said to be a symmetric pair, where $\mathbb{Z}_2\cong\langle\theta\rangle$ (cf. \cite{helgason1979differential}).
The triplet $(G,K,\langle\sigma,\tau\rangle)$ is a pair of Lie groups corresponding to a symmetric triad $(\mathfrak{g},\sigma,\tau)$, mentioned in section~\ref{sec. Gamma-symmetric Lie algebra} (cf. \cite{ikawa2025intersectionrealflagmanifolds}). 
Moreover, $(G,K_1,K_2)$ is said to be a \textit{symmetric triad}, where $K_1$ is a closed subgroup of $G$ satisfying $F_0(\sigma,G)\subset K_1\subset F(\sigma,G)$ and $K_2$ is a closed subgroup of $G$ satisfying $F_0(\tau,G)\subset K_2\subset F(\tau,G)$ (cf. \cite{ohno2025polarsantipodalsetsgeneralized}). 
Here, $F_0(\gamma,G)$ denotes the connected component of $F(\gamma,G)$ containing the identity element of $G$.

A $\Gamma$-symmetric space $M$ is constructed from a $\Gamma$-symmetric triple $(G,K,\Gamma)$ (not necessarily $\Gamma\cong\mathbb{Z}_2$ or $\mathbb{Z}_2\times\mathbb{Z}_2$) (cf. \cite{Quast&Sakai}).
The manifold $M$ and the $\Gamma$-symmetric structure  $\mu=\{\mu^{\gamma}\}_{\gamma\in\Gamma}$ are given by 
\begin{equation*}
    M:=G/K,\;\;\mu^{\gamma}(x,y):=a_x\gamma(a_x^{-1}b_y)K(=\gamma_x(y)),
\end{equation*}
where $a_x\in G$ (resp. $b_y\in G$) satisfies $x=a_x K$ (resp. $y=b_y K$).
This construction is respectively commutative with the direct product and the isomorphic relation.
If $(\varphi,\Phi)$ is an isomorphism from $(G,K,\Gamma)$ to $(G',K',\Gamma')$, then the isomorphism $(\overline{\varphi},\Phi)$ from $(G/K,\Gamma)$ to $(G'/K',\Gamma')$ is given by 
\begin{equation*}
    \overline{\varphi}:G/K\to G'/K' ; gK\mapsto \varphi(g)K',
\end{equation*}
which is well-defined from (B).

\subsection{Riemannian $\Gamma$-symmetric triples}\label{Riemannian Gamma-symmetric triples}
For given a $\Gamma$-symmetric triple $(G,K,\Gamma)$, suppose that $\Gamma\cong\mathbb{Z}_2$ or $\mathbb{Z}_2\times\mathbb{Z}_2$.
The group $\Gamma$ is regarded as a subgroup of $\Aut(G)$ via $\rho_G:\Gamma\to\Aut(G)$.
We define $\mathfrak{g}=\Lie G$ and $\mathfrak{k}=\Lie K$.
To simplify notation, for given an automorphism $\phi$ of $G$ (resp. a subgroup $\Phi$ of $\Aut(G)$) we use the same letter $\phi$ (resp. $\Phi$) for the differential of $\phi$ (resp. the group of differentials of all $\phi\in\Phi$) throughout this paper.
Recall the notations (\ref{eigenspace}) and (\ref{simultaneous eigenspace}) for eigenspaces.
Let $\sigma$ and $\tau$ be commutating elements of the automorphism group $\Aut(\mathfrak{g})$ of $\mathfrak{g}$ satisfying $\Gamma=\langle\sigma,\tau\rangle$. 
Since $\Gamma$ is finite and abelian, these automorphisms are simultaneously diagonalizable. 
Therefore (\ref{simultaneous eigenspace}) gives the following eigendecomposition of $\mathfrak{g}$:
\begin{equation}\label{eigendecomposition} \mathfrak{g}=\mathfrak{g}^{(+1,+1)}\oplus\mathfrak{g}^{(+1,-1)}\oplus\mathfrak{g}^{(-1,+1)}\oplus\mathfrak{g}^{(-1,-1)}.
\end{equation}
The symbols $(+1,-1)$, $(-1,+1)$ and $(-1,-1)$ depend on the choice of $\sigma$ and $\tau$, however, the decomposition (\ref{eigendecomposition}) is independent of it.
The condition $F_0(\Gamma,G)\subset K\subset F(\Gamma,G)$ implies $\mathfrak{k}=\mathfrak{g}^{(+1,+1)}$. 
We often write
\begin{equation}\label{definition of m}
\mathfrak{m}=\mathfrak{g}^{(+1,-1)}\oplus\mathfrak{g}^{(-1,+1)}\oplus\mathfrak{g}^{(-1,-1)}
\end{equation}
and call
\begin{equation}\label{decomposition of g by k and m}
    \mathfrak{g}=\mathfrak{k}\oplus\mathfrak{m}
\end{equation}
the \textit{standard decomposition}. 
The subspace $\mathfrak{m}$ is said to be a \textit{standard complement}.
By the definition of eigenspaces, the decomposition (\ref{eigendecomposition}) has a $\Gamma$-graded structure, that is, $[\mathfrak{g}^{a},\mathfrak{g}^{b}]\subset\mathfrak{g}^{ab}$ hold for all $a,b\in\{(+1,+1),(+1,-1),(-1,+1),(-1,-1)\}$.
Hence, the standard decomposition is reductive, i.e., $[\mathfrak{k},\mathfrak{m}]\subset\mathfrak{m}$. 

For simplicity, we use the next symbols throughout this paper.
Let $\Gamma^{\times}$ be the set of elements in $\Gamma$ excluding the identity $e\in\Gamma$.
We put $e:=(+1,+1)$, $b:=(+1,-1)$, $c:=(-1,+1)$ and $d:=(-1,-1)$.
When $\Gamma\cong\mathbb{Z}_2\times\mathbb{Z}_2$, we often treat $\Gamma$ as $\{e,b,c,d\}$ and $\Gamma^{\times}$ as $\{b,c,d\}$. 
In addition, we write $\GL(V)$ for the Lie group of all linear automorphisms of a vector space $V$.

Now, assume that a $\Gamma$-symmetric triple $(G,K,\Gamma)$ satisfies 
\renewcommand{\theenumi}{(\roman{enumi})}
\begin{enumerate}
    \item \label{compactivity of Adjoint action of isotropy group} $\Ad_G(K)$ is a compact Lie subgroup of $\GL(\mathfrak{g})$.
\end{enumerate}
Then the compact Lie group $\Ad_G(K)|_{\mathfrak{g}^{a}}$ acts on $\mathfrak{g}^{a}$ for each $a\in\Gamma^{\times}$ 
and there exists an $\Ad_G(K)$-invariant inner product $g^a$ on $\mathfrak{g}^a$, that is, $g^a$ satisfies 
\begin{enumerate}[resume]
    \item \label{ga is Ad(K)-inv.} $g^a(\Ad_G(k)X,\Ad_G(k)Y)=g^a(X,Y) \quad(\forall k\in K, \forall X,Y\in\mathfrak{g}^a)$.
\end{enumerate}
We define an $\Ad_G(K)$-invariant inner product $g$ on $\mathfrak{m}$ by $g:=\bigoplus_{a\in\Gamma^{\times}} g^a$ (orthogonal direct sum). 
Then $g$ satisfies
\begin{enumerate}[resume]
    \item \label{deco. is perp w.r.t. g} The decomposition (\ref{definition of m}) is orthogonal with respect to $g$.
\end{enumerate}
It is easy to check that \ref{deco. is perp w.r.t. g} is equivalent that $\gamma|_{\mathfrak{m}}:\mathfrak{m}\to\mathfrak{m}$ is an isometry regarding to $g$ for every $\gamma\in\Gamma$.
\begin{definition}
    $(G,K,\Gamma,g)$ satisfying \ref{compactivity of Adjoint action of isotropy group}-\ref{deco. is perp w.r.t. g} is called a \textit{Riemannian $\Gamma$-symmetric triple}.
\end{definition}
\begin{definition}\label{def of isom of Riem triples}
    $(G,K,\Gamma,g)$ and $(G',K',\Gamma',g')$ are said to be \textit{isomorphic} if there exist a Lie group isomorphism $\varphi:G\to G'$ and a group isomorphism $\Phi:\Gamma\to\Gamma'$ such that (A), (B) and
    \begin{equation*}
        \mathrm{(C)}\quad  g'(\varphi X,\varphi Y)=g(X,Y) \quad\mathrm{for\;all}\;X,Y\in\mathfrak{m}.
    \end{equation*}
    Then we write $(G,K,\Gamma,g)\cong(G',K',\Gamma',g')$ and $(\varphi,\Phi)$ is also called an \textit{isomorphism}.
\end{definition}
Note that $\varphi\mathfrak{g}^a=\mathfrak{g}'^{\Phi(a)}$ hold for all $a\in\Gamma^{\times}$ from (A) and (B) and hence $\varphi\mathfrak{m}=\mathfrak{m}'$.

Let $(G_i,K_i,\Gamma_i,g_i)$, $i=1,2$, be Riemannian $\Gamma_i$-symmetric triples
and $(G_1\times G_2,K_1\times K_2,\Sigma)$ a direct product such that $\Sigma\cong\mathbb{Z}_2$ or $\mathbb{Z}_2\times\mathbb{Z}_2$.
Then $(G_1\times G_2,K_1\times K_2,\Sigma,g_1\oplus g_2)$ is a Riemannian $\Sigma$-symmetric triple. 
This is also called a \textit{direct product}.

For given a Riemannian $\Gamma$-symmetric triple $(G,K,\Gamma,g)$, let $(G/K,\Gamma)$ be the $\Gamma$-symmetric space constructed from $(G,K,\Gamma)$. 
Identifying the tangent space at the origin of $G/K$ with the standard complement $\mathfrak{m}$, a $G$-invariant metric on $(G/K,\Gamma)$ is induced by $g$. 
We also use the same letter $g$ for the metric. 
It is well-known that $\Ad_G(K)$-invariant inner products on $\mathfrak{m}$ have a one-to-one correspondence with $G$-invariant metrics on $G/K$. 
By the next lemma, $(G/K,\Gamma)$ is a Riemannian $\Gamma$-symmetric space with respect to the metric $g$.
\begin{lemma}[cf. \cite{Bouyakoub&Goze&Remm}]
    For $(G,K,\Gamma,g)$ and $(G/K,\Gamma)$, the following conditions are equivalent:
    \begin{itemize}
        \item The decomposition (\ref{definition of m}) is orthogonal with respect to $g$.
        \item Every symmetry of $(G/K,\Gamma)$ is isometry with respect to $g$. 
    \end{itemize}
\end{lemma}
This construction is commutative with each isomorphic relation.
\section{$\Gamma$-symmetric Lie algebras}\label{sec. Gamma-symmetric Lie algebra}
\subsection{$\Gamma$-symmetric Lie algebras}
Suppose that $\Gamma\cong\mathbb{Z}_2$ or $\mathbb{Z}_2\times\mathbb{Z}_2$ and use the notation (\ref{simultaneous eigenspace})-(\ref{decomposition of g by k and m}) in this subsection.
Let $\mathfrak{g}$ be a Lie algebra over $\mathbb{R}$ with finite dimension and $\rho:\Gamma\to\Aut(\mathfrak{g})$ an injective homomorphism. 
We simply write $\Gamma$ instead of $\rho(\Gamma)$. 
Then $(\mathfrak{g},\Gamma)$ is called an \textit{$\Gamma$-symmetric Lie algebra}. 
We define the \textit{standard decomposition} and the \textit{standard complement} from (\ref{definition of m})-(\ref{decomposition of g by k and m}).
\begin{definition}
     $(\mathfrak{g},\Gamma)$ and $(\mathfrak{g}',\Gamma')$ are \textit{isomorphic}  if there exist a Lie algebra isomorphism $\varphi : \mathfrak{g}\to \mathfrak{g}'$ and a group isomorphism $\Phi : \Gamma \to \Gamma'$ satisfying
     \begin{equation*}
     \mathrm{(A)'}\quad\varphi\circ\gamma=\Phi(\gamma)\circ\varphi\textrm{\;\;for all } \gamma\in\Gamma.
     \end{equation*}
Then, we write $(\mathfrak{g},\Gamma)\cong(\mathfrak{g}',\Gamma')$ and call the pair $(\varphi,\Phi)$ an \textit{isomorphism}.
\end{definition}
Let $\sigma$ and $\tau$ be commutating elements of $\Aut(\mathfrak{g})$ generating $\Gamma$. 
Then, we also write $(\mathfrak{g},\sigma,\tau)$ for $(\mathfrak{g},\Gamma)$.
If $\Gamma\cong\mathbb{Z}_2\times\mathbb{Z}_2$, then $(\mathfrak{g},\sigma,\tau)$ is said to be a \textit{commutative symmetric triad}, which has been studied in the sense of Hermann actions (cf. \cite{ikawa2025intersectionrealflagmanifolds}). 
If $\Gamma=\langle\theta\rangle\cong\mathbb{Z}_2$, then $(\mathfrak{g},\theta)$ is said to be a \textit{symmetric Lie algebra} (cf. \cite{helgason1979differential}).

\begin{definition}
     Let $\Sigma$ be a subgroup of $\Gamma$. 
     A pair $(\mathfrak{g},\Sigma)$ is a \textit{$\Sigma$-subsymmetric Lie algebra} of a $\Gamma$-symmetric Lie algebra $(\mathfrak{g},\Gamma)$ if $\bigcap_{\theta\in\Sigma}\mathfrak{g}^{\theta}=\bigcap_{\theta\in\Gamma}\mathfrak{g}^{\theta}$, that is, $\bigcap_{\theta\in\Sigma}\mathfrak{g}^{\theta}=\mathfrak{k}$.
     The $\Sigma$-subsymmetric Lie algebra is called \textit{trivial} if $\Sigma=\Gamma$.
     Moreover, $(\mathfrak{g},\Gamma)$ is called \textit{proper} if it has no nontrivial $\Sigma$-subsymmetric Lie algebra for any subgroup $\Sigma$ of $\Gamma$. 
\end{definition}
Let $(G,K,\Gamma)$ be a $\Gamma$-symmetric triple and $\mathfrak{g}$ denotes the Lie algebra of $G$.
Then, $(\mathfrak{g},\Gamma)$ is the $\Gamma$-symmetric Lie algebra and this $(\mathfrak{g},\Gamma)$ is called a $\Gamma$-symmetric Lie algebra associated with $(G,K,\Gamma)$. 
If $(G,K,\Sigma)$ is a $\Sigma$-subsymmetric triple of $(G,K,\Gamma)$, then the $\Sigma$-symmetric Lie algebra $(\mathfrak{g},\Sigma)$ associated with $(G,K,\Sigma)$ is a $\Sigma$-subsymmetric Lie algebra of $(\mathfrak{g},\Gamma)$. 
In contrast, if $(\mathfrak{g},\Sigma)$ is a $\Sigma$-subsymmetric Lie algebra of $(\mathfrak{g},\Gamma)$, then $(G,K,\Sigma)$ is a $\Sigma$-subsymmetric triple of $(G,K,\Gamma)$ since $F_0(\Sigma,G)=F_0(\Gamma,G)\subset K\subset F(\Gamma,G)\subset F(\Sigma,G)$ holds.

\begin{lemma}\label{condition that Z_2timesZ_2-symmetric Lie algebras have subsymmetric Lie algebras}
Assume that $\Gamma\cong\mathbb{Z}_2\times\mathbb{Z}_2$. 
    A $\Gamma$-symmetric Lie algebra $(\mathfrak{g},\Gamma)$ is proper if and only if 
    all of the vector spaces $\mathfrak{g}^{b}, \mathfrak{g}^{c}$ and $\mathfrak{g}^{d}$ are nonzero vector spaces.
\end{lemma}
\begin{proof}
We prove both directions by contraposition.
Suppose that $(\mathfrak{g},\Gamma)$ has a nontrivial $\Sigma$-subsymmetric Lie algebra $(\mathfrak{g},\Sigma)$.
We write $\Gamma=\langle\sigma,\tau\rangle$ and $\Sigma=\langle\theta\rangle$.
Then $\mathfrak{g}^{\theta}$ always satisfies $\mathfrak{g^{\theta}}=\mathfrak{k}\oplus\mathfrak{g}^{b}$, $\mathfrak{k}\oplus\mathfrak{g}^{c}$ or $\mathfrak{k}\oplus\mathfrak{g}^{d}$. 
If $\mathfrak{g^{\theta}}=\mathfrak{k}\oplus\mathfrak{g}^{b}$, then we obtain $\mathfrak{g}^{b}=\{0\}$ since $\mathfrak{g}^{\theta}=\mathfrak{k}$.
The same proof works for other cases.
Conversely, assume that $\mathfrak{g}^{b}=\{0\}$, $\mathfrak{g}^{c}=\{0\}$ or $\mathfrak{g}^{d}=\{0\}$. 
If $\mathfrak{g}^{b}=\{0\}$, then $\mathfrak{g}^{\sigma}=\mathfrak{k}$, implying that $(\mathfrak{g},\langle\sigma\rangle)$ is a nontrivial $\langle\sigma\rangle$-subsymmetric Lie algebra of $(\mathfrak{g},\Gamma)$.
The same proof works for others.
\end{proof}

\begin{definition}\label{Def of a direct sum of Gamma-symmetric Lie algebras}
    Let $(\mathfrak{g}_i,\Gamma_i)$, $i=1,2$, be $\Gamma_i$-symmetric Lie algebras and $\Sigma$ a subgroup of $\Gamma_1\times\Gamma_2$ with $\Sigma\cong\mathbb{Z}_2$ or $\mathbb{Z}_2\times\mathbb{Z}_2$.
If $(\mathfrak{g}_1\oplus\mathfrak{g}_2,\Sigma)$ is a $\Sigma$-subsymmetric Lie algebra of $(\mathfrak{g}_1\oplus\mathfrak{g}_2,\Gamma_1\times\Gamma_2)$, then $(\mathfrak{g}_1\oplus\mathfrak{g}_2,\Sigma)$ is called a \textit{direct sum} of $(\mathfrak{g}_1,\Gamma_1)$ and $(\mathfrak{g}_2,\Gamma_2)$.
\end{definition}
\subsection{Riemannian $\Gamma$-symmetric Lie algebras}
Assume that a $\Gamma$-symmetric Lie algebra $(\mathfrak{g},\Gamma)$ satisfies 
\renewcommand{\theenumi}{(\roman{enumi})$'$}
\begin{enumerate}
    \item \label{compactivity of adjoint action of the Lie algebra k} $\ad_{\mathfrak{g}}(\mathfrak{k})$ generates a compact connected Lie subgroup of $\GL(\mathfrak{g})$.
\end{enumerate}
A $\Gamma$-symmetric Lie algebra $(\mathfrak{g},\Gamma)$ satisfying (i)$'$ is called an \textit{orthogonal $\Gamma$-symmetric Lie algebra}. 
From (i)$'$, there exists an $\ad_{\mathfrak{g}}(\mathfrak{k})$-invariant inner product $g^a$ on $\mathfrak{g}^a$ for each $a\in\Gamma^{\times}$, that is, the inner product $g^a$ satisfies
\begin{enumerate}[resume]
    \item \label{ga is ad(k)-inv.} $g^a(\ad_{\mathfrak{g}}(Z)X,Y)+g^a(X,\ad_{\mathfrak{g}}(Z)Y)=0
        \quad (\forall Z\in \mathfrak{k},\forall X,Y\in\mathfrak{g}^a)$.
\end{enumerate}
We define an $\ad_{\mathfrak{g}}(\mathfrak{k})$-invariant inner product $g$ on the standard complement $\mathfrak{m}$ by $g:=\bigoplus_{a\in\Gamma^{\times}} g^a$ (orthogonal direct sum).
Then $g$ satisfies
\begin{enumerate}[resume]
    \item \label{deco.' is perp w.r.t. g} The decomposition (\ref{definition of m}) is orthogonal with respect to $g$.
\end{enumerate}
\begin{definition}
    $(\mathfrak{g},\Gamma,g)$ satisfying \ref{compactivity of adjoint action of the Lie algebra k}-\ref{deco.' is perp w.r.t. g} is called a \textit{Riemannian $\Gamma$-symmetric Lie algebra}.
\end{definition}
\begin{definition}
    $(\mathfrak{g},\Gamma,g)$ and $(\mathfrak{g}',\Gamma',g')$ are said to be \textit{isomorphic} if there exist a Lie algebra isomorphism $\varphi:\mathfrak{g}\to \mathfrak{g}'$ and a group isomorphism $\Phi:\Gamma\to\Gamma'$ such that (A)$'$ and
    \begin{equation*}
        \mathrm{(C)'}\quad  g'(\varphi X,\varphi Y)=g(X,Y) \quad\mathrm{for\;all}\;X,Y\in\mathfrak{m}.
    \end{equation*}
    Then we write $(\mathfrak{g},\Gamma,g)\cong(\mathfrak{g}',\Gamma',g')$ and $(\varphi,\Phi)$ is called an \textit{isomorphism}.
\end{definition}

For given Riemannian $\Gamma_i$-symmetric Lie algebras $(\mathfrak{g}_i,\Gamma_i,g_i)$, $i=1,2$, let $(\mathfrak{g}_1\oplus\mathfrak{g}_2,\Sigma)$ be a direct sum of $(\mathfrak{g}_1,\Gamma_1)$ and $(\mathfrak{g}_2,\Gamma_2)$.
Then $(\mathfrak{g}_1\oplus\mathfrak{g}_2,\Sigma,g_1\oplus g_2)$ is a Riemannian $\Sigma$-symmetric Lie algebra. 
This is also called a \textit{direct sum} of $(\mathfrak{g}_1,\Gamma_1,g_1)$ and $(\mathfrak{g}_2,\Gamma_2,g_2)$. 
Similarly to $\Gamma$-symmetric spaces, every direct sum $(\mathfrak{g}_1\oplus\mathfrak{g}_2,\Sigma,g_1\oplus g_2)$ is written by generators of $\Gamma_i$ (see section~\ref{sec:Product structures}). 
\section{Decomposition of effective semisimple Riemmanian $\mathbb{Z}_2\times\mathbb{Z}_2$-symmetric Lie algebras of nontrivial type}\label{sec. decomposition of nontrivial type}
\subsection{Semisimple Riemannian $\mathbb{Z}_2\times\mathbb{Z}_2$-symmetric Lie algebras of nontrivial type}
In this subsection, assume that $\Gamma\cong\mathbb{Z}_2\times\mathbb{Z}_2$. 
We first consider a semisimple $\Gamma$-symmetric Lie algebra $(\mathfrak{g},\Gamma)$, that is, $\Gamma$-symmetric Lie algebra $(\mathfrak{g},\Gamma)$ such that $\mathfrak{g}$ is semisimple.
\begin{definition}\label{def of effective}
    A $\Gamma$-symmetric Lie algebra $(\mathfrak{g},\Gamma)$ is said to be \textit{effective} if there exists no nontrivial ideal of $\mathfrak{g}$ contained in $\mathfrak{k}$. Also, a Riemannian $\Gamma$-symmetric Lie algebra $(\mathfrak{g},\Gamma,g)$ is said to be \textit{effective} if $(\mathfrak{g},\Gamma)$ is effective.
\end{definition}
\begin{lemma}\label{property_of_semisimple_Riemannian_Gamma-symmetric_Lie_algebra}
    Let $(\mathfrak{g},\Gamma)$ be an orthogonal $\Gamma$-symmetric Lie algebra and $B$ the Killing form on $\mathfrak{g}$. 
    Then,
    \renewcommand{\theenumi}{(\Roman{enumi})}
    \begin{enumerate}[series=lemma-semisimple]
        \item \label{B-perp}
        $B(\mathfrak{g}^a,\mathfrak{g}^b)=\{0\}$ for any distinct $a,b\in\Gamma$, in particular, $B(\mathfrak{k},\mathfrak{m})=\{0\}$. 
        \item \label{B-negative deginite on k}
        If $(\mathfrak{g},\Gamma)$ is effective, then $B_{\mathfrak{k}}:=B|_{\mathfrak{k}\times\mathfrak{k}}$ is negative definite on $\mathfrak{k}$.  
        \item \label{The form of k and m for semisimple Gamma-symmetric Lie albenras}
        If $(\mathfrak{g},\Gamma)$ is effective and $\mathfrak{g}$ is semisimple, then $\mathfrak{k}=\sum_{a\in\Gamma^{\times}}[\mathfrak{g}^a,\mathfrak{g}^a]$ and $\mathfrak{m}\neq\{0\}$.
        \item \label{nondegenerancy of B on each eigenspades}
        If $\mathfrak{g}$ is semisimple, then $B$ is nondegenerate on $\mathfrak{g}^a$ for each $a\in\Gamma$.
    \end{enumerate}
\end{lemma}
\begin{proof}
\ref{B-perp}
    For any distinct $a,a'\in\Gamma$ with $a'\neq e$, there exists $\gamma\in\Gamma$ such that $\gamma X=X$ and $\gamma Y=-Y$ for all $X\in\mathfrak{g}^a$ and $Y\in\mathfrak{g}^{a'}$. 
    Thus $B(X,Y)=B(\gamma X,\gamma Y)=-B(X,Y)$, and hence $B(X,Y)=0$.
    
\ref{B-negative deginite on k} From \ref{compactivity of adjoint action of the Lie algebra k}, there exists an $\ad_{\mathfrak{g}}(\mathfrak{k})$-invariant inner product $\langle,\rangle$ on $\mathfrak{g}$.
    Therefore, for each $Z\in\mathfrak{k}$, eigenvalues of $\ad_{\mathfrak{g}}(Z)$ are written as $\sqrt{-1}\lambda_i$, $\lambda_i\in\mathbb{R}$.
    Hence, $B(Z,Z)=-\sum_{i}\lambda_i^2\leq0$. 
    This implies that $B(Z,Z)=0$ is equivalent to $Z\in\mathfrak{k}\cap\mathfrak{z}$ for the center $\mathfrak{z}$ of $\mathfrak{g}$. 
    On the other hand, $\mathfrak{k}\cap\mathfrak{z}=\{0\}$ since $(\mathfrak{g},\Gamma)$ is effective.
    Hence, $B(Z,Z)=0$ is equivalent to $Z=0$. 

\ref{The form of k and m for semisimple Gamma-symmetric Lie albenras}
Put $\mathfrak{k}_0:=\sum_{a\in\Gamma^{\times}}[\mathfrak{g}^a,\mathfrak{g}^a]$ and 
$\mathfrak{k}_1:=\{X\in\mathfrak{k}\mid B(X,\mathfrak{k}_0)=\{0\}\}$.
From \ref{B-negative deginite on k}, we have $\mathfrak{k}=\mathfrak{k}_0\oplus\mathfrak{k}_1$. 
Since \ref{B-perp} and $B$ is $\ad_{\mathfrak{g}}(\mathfrak{g})$-invariant and nondegenerate, we have $[\mathfrak{k}_1,\mathfrak{g}^a]=\{0\}$ for all $a\in\Gamma^{\times}$, which imply $[\mathfrak{k}_1,\mathfrak{m}]=\{0\}$, $[\mathfrak{k}_1,\mathfrak{k}_0]=\{0\}$ and $[\mathfrak{k}_1,\mathfrak{k}]\subset\mathfrak{k}_1$.
Therefore, $\mathfrak{k}_1$ is an ideal of $\mathfrak{g}$ contained in $\mathfrak{k}$. 
Since $(\mathfrak{g},\Gamma)$ is effective, we obtain $\mathfrak{k}_1=\{0\}$.
Next, if $\mathfrak{m}=\{0\}$, then $\mathfrak{g}=\{0\}$ because $\mathfrak{k}=\sum_{a\in\Gamma^{\times}}[\mathfrak{g}^a,\mathfrak{g}^a]=\{0\}$. 
Thus $\mathfrak{m}\neq\{0\}$.

\ref{nondegenerancy of B on each eigenspades}
If $X\in\mathfrak{g}^a$ satisfies $B(X,\mathfrak{g}^a)=\{0\}$, then $B(X,\mathfrak{g})=\{0\}$ holds by \ref{B-perp}.
Since $\mathfrak{g}$ is semisimple, we have $X=0$, which completes the proof.
\end{proof}

\begin{remark}
In contrast to the result in \cite[Lemma 2.1]{Goze&Remm}, we find that the sum $\mathfrak{k}=\sum_{a\in\Gamma^{\times}}[\mathfrak{g}^a,\mathfrak{g}^a]$ is generally not a direct sum.
A counterexample is given by a flag manifold.
In example~\ref{flag manifolds}, we consider the case $(r_1,r_2,r_3,r_4)=(1,2,2,0)$.
We define
\[
B:=
\left(
\begin{array}{c|cc|cc}
0 & 0 & 0 & 2 & 1\\
\hline
0 & 0 & 0 & 0 & 0\\
0 & 0 & 0 & 0 & 0\\ 
\hline
-2 & 0 & 0 & 0 & 0\\ 
-1 & 0 & 0 & 0 & 0
\end{array}
\right),\;
B':=
\left(
\begin{array}{c|cc|cc}
0 & 0 & 0 & 1 & 1\\
\hline
0 & 0 & 0 & 0 & 0\\
0 & 0 & 0 & 0 & 0\\ 
\hline
-1 & 0 & 0 & 0 & 0\\ 
-1 & 0 & 0 & 0 & 0
\end{array}
\right)
\in\mathfrak{g}^c
\]
and
\[
C:=
\left(
\begin{array}{c|cc|cc}
0 & 0 & 0 & 0 & 0\\
\hline
0 & 0 & 0 & 2 & 1\\
0 & 0 & 0 & 0 & 1\\ 
\hline
0 & -2 & 0 & 0 & 0\\ 
0 & -1 & -1 & 0 & 0
\end{array}
\right),\;
C':=
\left(
\begin{array}{c|cc|cc}
0 & 0 & 0 & 0 & 0\\
\hline
0 & 0 & 0 & 1 & 1\\
0 & 0 & 0 & 0 & 1\\ 
\hline
0 & -1 & 0 & 0 & 0\\ 
0 & -1 & -1 & 0 & 0
\end{array}
\right)
\in\mathfrak{g}^d.
\]
Then, we can easily calculate
\[
[B,B']=[C,C']=
\left(
\begin{array}{c|cc|cc}
0 & 0 & 0 & 0 & 0\\
\hline
0 & 0 & 0 & 0 & 0\\
0 & 0 & 0 & 0 & 0\\ 
\hline
0 & 0 & 0 & 0 & -1\\ 
0 & 0 & 0 & 1 & 0
\end{array}
\right)
\in\mathfrak{k}\setminus\{0\},
\]
which means that $[\mathfrak{g}^c,\mathfrak{g}^c]\cap[\mathfrak{g}^d,\mathfrak{g}^d]\neq\{0\}$.
\end{remark}
\begin{definition}\label{def of irreducible}
    A $\Gamma$-symmetric Lie algebra $(\mathfrak{g},\Gamma)$ is said to be \textit{$\Gamma$-irreducible} (or \textit{irreducible}, for short) if there exists no nontrivial ideal $\mathfrak{g}_0$ of $\mathfrak{g}$ satisfying $\gamma\mathfrak{g}_0\subset\mathfrak{g}_0$ for all $\gamma\in\Gamma$. A $\Gamma$-symmetric Lie algebra is said to be \textit{$\Gamma$-reducible} (or \textit{reducible}) if it is not $\Gamma$-irreducible. 
    Moreover, a Riemannian $\Gamma$-symmetric Lie algebra $(\mathfrak{g},\Gamma,g)$ is said to be $\Gamma$-irreducible (resp. $\Gamma$-reducible) if $(\mathfrak{g},\Gamma)$ is $\Gamma$-irreducible (resp. $\Gamma$-reducible).
\end{definition}
Goze and Remm gave an alternative definition of $\Gamma$-irreducibility regarding the infinitesimal isotropy representation of $\mathfrak{k}$ on $\mathfrak{m}$ in \cite[Definition 2.2]{Goze&Remm}.
The next lemma is not only used in the proof of theorem~\ref{decomposition of semisimple Lie algebra of nontrival type} but also suggests the relationship between the two $\Gamma$-irreducibilities.

\begin{lemma}\label{lemma_gammareducible}
Let $(\mathfrak{g},\Gamma)$ be an effective semisimple orthogonal $\Gamma$-symmetric Lie algebra. 
The following conditions are equivalent:
    \renewcommand{\theenumi}{(\Roman{enumi})}
    \begin{enumerate}[resume*=lemma-semisimple]
        \item \label{Gamma-reducible} $(\mathfrak{g},\Gamma)$ is $\Gamma$-reducible.
        \item \label{Decomposition of infinitesimal isotropy representation}
        There exist nontrivial vector subspaces $\mathfrak{m}_0$ and $\mathfrak{m}_1$ of $\mathfrak{m}$ such that 
        \begin{itemize}
            \item $\mathfrak{m}=\mathfrak{m}_0\oplus\mathfrak{m}_1$ and $[\mathfrak{m}_0,\mathfrak{m}_1]=\{0\}$,
            \item $\mathfrak{k}\oplus\mathfrak{m}_i$ is a $\Gamma_i$-graded Lie algebra for each $i=0,1$, 
where $\Gamma_i$ is a nontrivial subgroup of $\Gamma$.
        \end{itemize}
    \end{enumerate}
\end{lemma}
\begin{proof}
Assume that \ref{Gamma-reducible}. 
Then, since $\mathfrak{g}$ is semisimple, there exist nontrivial ideals $\mathfrak{g}_i$, $i=0,1$, of $\mathfrak{g}$ such that $\mathfrak{g}=\mathfrak{g}_0\oplus\mathfrak{g}_1$ (a direct sum of ideals) and $\gamma\mathfrak{g}_i\subset\mathfrak{g}_i$ for all $\gamma\in\Gamma$. 
We define 
    \begin{equation}\label{decomposition_of_mi_e.t.c.}
        \mathfrak{k}_i:=\mathfrak{k}\cap\mathfrak{g}_i,\quad 
        \mathfrak{g}^a_i:=\mathfrak{g}^a\cap\mathfrak{g}_i,\quad         
    \mathfrak{m}_i:=\bigoplus_{a\in\Gamma^{\times}}\mathfrak{g}^a_i
    \textrm{\;\; and \;\;}
    \Gamma_i:=\{\gamma|_{\mathfrak{g}_i} \mid \gamma\in\Gamma\}.
    \end{equation}
Each $\Gamma_i$ is a nontrivial group since $(\mathfrak{g},\Gamma)$ is effective.
Then, we obtain
    \begin{equation}\label{g_i=k_i+m_i,k=k_0+k_1,m=m_0+m_1}
        \mathfrak{g}_i=\mathfrak{k}_i\oplus\mathfrak{m}_i,\quad
        \mathfrak{k}=\mathfrak{k}_0\oplus\mathfrak{k}_1,\quad
        \mathfrak{m}=\mathfrak{m}_0\oplus\mathfrak{m}_1
        \textrm{\quad and \quad} 
        [\mathfrak{m}_0,\mathfrak{m}_1]=\{0\}.
    \end{equation}
Moreover, $\mathfrak{k}\oplus\mathfrak{m}_i$ are $\Gamma_i$-graded Lie algebras and $(\mathfrak{g}_i,\Gamma_i)$ are $\Gamma_i$-symmetric Lie algebras.
Since $(\mathfrak{g},\Gamma)$ is effective, $(\mathfrak{g}_i,\Gamma_i)$ are effective.
Hence we obtain $\mathfrak{m}_i\neq\{0\}$, $i=0,1$, meaning $\mathfrak{m}_i$ are nontrivial from \ref{The form of k and m for semisimple Gamma-symmetric Lie albenras}.

Conversely, assume that \ref{Decomposition of infinitesimal isotropy representation}. 
We define vector subspaces $\mathfrak{g}_i$ of $\mathfrak{g}$ by
\begin{equation*}
    \mathfrak{g}_i:=\mathfrak{k}_i\oplus\mathfrak{m}_i,\textrm{\quad where \quad}
    \mathfrak{k}_i:=\sum_{a\in\Gamma^{\times}}[\mathfrak{g}^a_i,\mathfrak{g}^a_i]
    \textrm{\quad and \quad}
    \mathfrak{g}^a_i:=\mathfrak{g}^a\cap\mathfrak{m}_i
\end{equation*}
for each $a\in\Gamma^{\times}$. 
From \ref{The form of k and m for semisimple Gamma-symmetric Lie albenras}, we have $\mathfrak{k}=\mathfrak{k}_0+\mathfrak{k}_1$.
Since we can verify that $\mathfrak{k}_1$ coincides with  the perpendicular complement of $\mathfrak{k}_0$ in $\mathfrak{k}$ with respect to the Killing form $B$, we obtain $\mathfrak{k}=\mathfrak{k}_0\oplus\mathfrak{k}_1$.
Moreover, $[\mathfrak{m}_0,\mathfrak{m}_1]=\{0\}$ implies that $[\mathfrak{k}_0,\mathfrak{m}_1]=[\mathfrak{k}_1,\mathfrak{m}_0]=[\mathfrak{k}_0,\mathfrak{k}_1]=\{0\}$ and thus $[\mathfrak{g}_0,\mathfrak{g}_1]=\{0\}$.
Hence, $\mathfrak{g}_i$ are ideals of $\mathfrak{g}$ and we obtain $\mathfrak{g}=\mathfrak{g}_0\oplus\mathfrak{g}_1$.
The ideals $\mathfrak{g}_i$ are nontrivial since $\mathfrak{m}_i$ are nontrivial.
In addition, 
$\gamma\mathfrak{g}_i\subset\mathfrak{g}_i$ hold for all $\gamma\in\Gamma$
since $\gamma$ acts as a scalar map on each eigenspace.
We have proved the equivalence of \ref{Gamma-reducible} and \ref{Decomposition of infinitesimal isotropy representation}.
\end{proof}

Let $X_{\mathfrak{k}}$ (resp. $X_{\mathfrak{m}}$) denote the $\mathfrak{k}$-component (resp. $\mathfrak{m}$-component) of $X\in\mathfrak{g}$ with respect to the standard decomposition \eqref{decomposition of g by k and m} and set $V_{\mathfrak{k}}=\{X_{\mathfrak{k}}\mid X\in V\}$ and $V_{\mathfrak{m}}=\{X_{\mathfrak{m}}\mid X\in V\}$ for any vector subspace $V$ of $\mathfrak{g}$.
The following definition provides the sufficient condition for existence of the decomposition of a semisimple Riemannian $\Gamma$-symmetric Lie algebra.
\begin{definition}\label{def-nontrivial type}
An effective semisimple Riemannian $\Gamma$-symmetric Lie algebra $(\mathfrak{g}, \Gamma, g)$ is said to be of \textit{nontrivial type} if $\ad_{\mathfrak{g}}([\mathfrak{m},\mathfrak{m}]_{\mathfrak{k}})X=\{0\}$ implies $X=0$ for any $X\in\mathfrak{m}$, that is, $[\mathfrak{m}, \mathfrak{m}]_{\mathfrak{k}}$ has no invariant vector in $\mathfrak{m}$.
\end{definition}
We state the geometrical meaning of nontrivial type.
As we will see in section~\ref{sec. the relationship between Riemannian Gamma-symmetric triples, Riemannian Gamma-symmetric Lie algebras and Riemannian Gamma-symmetric spaces}, we can construct Riemannian $\Gamma$-symmetric Lie algebra $(\mathfrak{g},\Gamma,g)$ from a given Riemannian $\Gamma$-symmetric space $(M,\Gamma,\mu,g)$.
Then, according to \cite[Chapter~X]{Kobayashi-Nomizu}, the reductive homogeneous space $M$ has a unique connection $\nabla^C$, called the \textit{canonical connection}, on $M$ which satisfies
\begin{equation}\label{tor and cur of cano}
T^{\nabla^C}_o(X,Y)=-[X,Y]_{\mathfrak{m}},
\qquad
R^{\nabla^C}_o(X,Y)Z=-\ad_{\mathfrak{g}}([X,Y]_{\mathfrak{k}})Z
\end{equation}
for all $X,Y,Z\in\mathfrak{m}$, where $T^{\nabla^C}_o$ and $R^{\nabla^C}_o$ denote the torsion and curvature tensors at the origin $o\in M$ with respect to $\nabla^C$, respectively.
Also, note that $\mathfrak{m}$ and $T_oM$ are identified under a map $\pi_*:\mathfrak{m}\ni X\mapsto X^{\sharp}_o=\left.\frac{d}{dt}\right|_{t=0}(\mathrm{exp}\,tX)(o)\in T_oM$.
Hence, the condition for nontrivial type means that the action of the Lie algebra of the holonomy group of the canonical connection has no nonzero invariant vector (cf. \cite[Corollary~4.3 Chapter~X]{Kobayashi-Nomizu}).
\begin{remark}

The canonical connection $\nabla^C$ and the Levi-Civita connection do not coincide in general.
Indeed, if they are equal, then $[\mathfrak{m},\mathfrak{m}]\subset\mathfrak{k}$ holds from \eqref{tor and cur of cano}.
Conversely, for given Riemannian symmetric space, its canonical connection and Levi-Civita connection coincide.
\end{remark}
A Riemannian symmetric Lie algebra $(\mathfrak{g},\mathbb{Z}_2,g)$ such that $[\mathfrak{m},\mathfrak{m}]_{\mathfrak{k}}(=[\mathfrak{m},\mathfrak{m}])$ has no invariant vector in $\mathfrak{m}$ is said to be of semisimple type (cf. \cite{ise1991lie}). 
It is known that an effective Riemannian symmetric Lie algebra $(\mathfrak{g},\mathbb{Z}_2,g)$ is of semisimple type if and only if $\mathfrak{g}$ is semisimple.
On the other hand, an effective semisimple Riemannian $\Gamma$-symmetric Lie algebra is not necessarily of nontrivial type (example~\ref{An example of a non-perpendicular decomposition}).

From the proof of lemma~\ref{lemma_gammareducible}, an effective semisimple $\Gamma$-symmetric Lie algebra $(\mathfrak{g}, \Gamma)$ is $\Gamma$-reducible if and only if it can be decomposed into a direct sum of two (positive-dimensional) effective semisimple $\Gamma$-symmetric Lie algebras (see \eqref{decomposition_of_mi_e.t.c.} and \eqref{g_i=k_i+m_i,k=k_0+k_1,m=m_0+m_1}).
However, for a given effective semisimple Riemannian $\Gamma$-symmetric Lie algebra $(\mathfrak{g}, \Gamma, g)$, even if it cannot be decomposed into a direct sum of two (positive-dimensional) effective semisimple Riemannian $\Gamma$-symmetric Lie algebras, its underlying 
$\Gamma$-symmetric Lie algebra $(\mathfrak{g}, \Gamma)$ can still be $\Gamma$-reducible (example~\ref{An example of a non-perpendicular decomposition}).
This situation does not occur if $(\mathfrak{g}, \Gamma, g)$ is of nontrivial type:
\begin{theorem}\label{decomposition of semisimple Lie algebra of nontrival type}
Let $(\mathfrak{g},\Gamma,g)$ be an effective semisimple Riemannian $\Gamma$-symmetric Lie algebra of nontrivial type. 
Then there exist a family of effective, semisimple and  $\Gamma_i$-irreducible Riemannian $\Gamma_i$-symmetric Lie algebras of nontrivial type $\{(\mathfrak{g}_i,\Gamma_i,g_i)\}_{i=1}^r$ and a direct sum of $(\mathfrak{g}_1,\Gamma_1,g_1), \ldots, (\mathfrak{g}_r,\Gamma_r,g_r)$, denoted by $(\mathfrak{g}_1,\Gamma_1,g_1)\oplus\cdots\oplus(\mathfrak{g}_r,\Gamma_r,g_r)$, such that
\begin{equation}\label{decom._of_semisimple_Riem_Lie_alg.}
(\mathfrak{g},\Gamma,g)\cong(\mathfrak{g}_1,\Gamma_1,g_1)\oplus\cdots\oplus(\mathfrak{g}_r,\Gamma_r,g_r).
\end{equation}
Furthermore, the family $\{(\mathfrak{g}_i,\Gamma_i,g_i)\}_{i=1}^r$ is unique up to isomorphism and permutation.
\end{theorem}
\begin{proof}
For the $\Gamma$-symmetric Lie algebra $(\mathfrak{g},\Gamma)$, let $(\mathfrak{g}_0,\Gamma_0)$ and $(\mathfrak{g}_1,\Gamma_1)$ be $\Gamma_i$-symmetric Lie algebras ($i=0,1$) defined by \eqref{decomposition_of_mi_e.t.c.} and \eqref{g_i=k_i+m_i,k=k_0+k_1,m=m_0+m_1}.
Then the subgroup $\Sigma=\{(\,\gamma|_{\mathfrak{g}_0},\,\gamma|_{\mathfrak{g}_1}\,)\mid\gamma\in\Gamma\}\subset\Gamma_0\times\Gamma_1$ gives the direct sum $(\mathfrak{g}_0\oplus\mathfrak{g}_1,\Sigma)$ isomorphic to the original $(\mathfrak{g},\Gamma)$.
Since $(\mathfrak{g},\Gamma)$ is effective (resp. orthogonal), each $(\mathfrak{g}_i,\Gamma_i)$ is effective (resp. orthogonal).
Since $\mathfrak{g}$ is finite-dimensional, by repeating this process, we obtain the family of effective, orthogonal, semisimple and  $\Gamma_i$-irreducible $\Gamma_i$-symmetric Lie algebras $\{(\mathfrak{g}_i,\Gamma_i)\}_{i=1}^r$ and the direct sum of $(\mathfrak{g}_1,\Gamma_1), \ldots, (\mathfrak{g}_r,\Gamma_r)$ isomorphic to $(\mathfrak{g},\Gamma)$.
Setting $g_i:=g|_{\mathfrak{m}_i\times\mathfrak{m}_i}$, each $(\mathfrak{g}_i,\Gamma_i,g_i)$ is Riemannian $\Gamma_i$-symmetric Lie algebra. 
Since $(\mathfrak{g},\Gamma,g)$ is of nontrivial type, each $(\mathfrak{g}_i,\Gamma_i,g_i)$ is also of nontrivial type.
From these arguments, to prove the existence of the family and the direct sum, it suffices to show that $\mathfrak{g}_1$ and $\mathfrak{g}_2$ are perpendicular with respect to $g$ when $r=2$.
In order to do this, we prove that $\mathfrak{g}_{1}^{a}$ and ${\mathfrak{g}}_{2}^{a}$ are perpendicular for each $a\in\Gamma^{\times}$. 
Note that $\mathfrak{g}^a=\mathfrak{g}^a_1\oplus\mathfrak{g}^a_2$.
We define ${\mathfrak{g}^a_1}^{\perp}$ by the perpendicular complement of $\mathfrak{g}^a_1$ in $\mathfrak{g}^a$ with respect to $g^a=g|_{\mathfrak{g}^a\times\mathfrak{g}^a}$ and then we have $\mathfrak{g}^a=\mathfrak{g}^a_1\oplus {\mathfrak{g}^a_1}^{\perp}$.
For $X\in \mathfrak{g}^a_2$, $X_1$ denotes $\mathfrak{g}^a_1$ component of $X$ with respect to the decomposition $\mathfrak{g}^a=\mathfrak{g}^a_1\oplus {\mathfrak{g}^a_1}^{\perp}$.
On the other hand, we obtain $g|_{\mathfrak{g}^a_1\times \mathfrak{g}^a_1}([Z,X_1],Y)=g|_{\mathfrak{g}^a\times \mathfrak{g}^a}([Z,X],Y)=0$ for all $Y\in\mathfrak{g}^a_1$ and for all $Z\in \mathfrak{k}_1$ because $[\mathfrak{k}_1,\mathfrak{g}^a_2]=\{0\}$ and ${\mathfrak{g}^a_1}^{\perp}$ is $\mathfrak{k}_1$-invariant.
Thus we have $[\mathfrak{k}_1,X_1]=\{0\}$.
Since $(\mathfrak{g}_1,\Gamma_1)$ is of nontrivial type, we obtain $X_1=0$, that is, $\mathfrak{g}^a_2\subset{\mathfrak{g}^a_1}^{\perp}$.
Here note that $[\mathfrak{m}_1,\mathfrak{m}_1]_{\mathfrak{k}_1}=\mathfrak{k}_1$ from \ref{The form of k and m for semisimple Gamma-symmetric Lie albenras} in lemma~\ref{property_of_semisimple_Riemannian_Gamma-symmetric_Lie_algebra}.
Comparing the dimensions, we get $\mathfrak{g}^a_2={\mathfrak{g}^a_1}^{\perp}$.

Next, we prove the uniqueness.
Since every direct sum is expressed by the generators of $\Gamma_{1}$, $\ldots$, $\Gamma_{r}$ as mentioned at the end of section~\ref{sec:Product structures}, we can take the generators $\sigma_{i}$ and $\tau_{i}$ of $\Gamma_{i}$ to satisfy 
\begin{equation}\label{sigma,tau}
\sigma=\sigma_{1}\times\cdots\times\sigma_{r}
\quad\text{and}\quad
\tau=\tau_{1}\times\cdots\times\tau_{r}
\end{equation}
for fixed generators $\sigma$, $\tau$ of $\Gamma$.
Let us denote another decomposition by 
\begin{equation*}\label{In proof, another decom.}
(\mathfrak{g},\Gamma,g)\cong(\mathfrak{g}'_1,\Gamma'_1,g'_1)\oplus\cdots\oplus(\mathfrak{g}'_{r'},\Gamma'_{r'},g'_{r'}).
\end{equation*}
Let $(\psi,\Psi)$ be an isomorphism from $(\mathfrak{g}'_1,\Gamma'_1,g'_1)\oplus\cdots\oplus(\mathfrak{g}'_{r'},\Gamma'_{r'},g'_{r'})$ to $(\mathfrak{g},\Gamma,g)$.
For all $i'\in\{1,\ldots,r'\}$, we identify $\mathfrak{g}'_{i'}$ with the image $\psi(\mathfrak{g}'_{i'})$, the automorphisms $\theta'_{i'}\in\Gamma'_{i'}$ with $\Psi(\theta'_{i'})$, and $g'_1\oplus\cdots\oplus g'_{r'}$ with the pushforward metric $\psi_*(g'_1)\oplus\cdots\oplus\psi_*(g'_{r'})(=g)$.
In the same manner as \eqref{sigma,tau}, we can take the generators $\sigma'_{i'}$ and $\tau'_{i'}$ of $\Gamma'_{i'}$ to satisfy $\sigma=\sigma'_{1}\times\cdots\times\sigma'_{r'}$ and $\tau=\tau'_{1}\times\cdots\times\tau'_{r'}$ for the generators $\sigma$, $\tau$ of $\Gamma$.
Based on this situation, what we want to prove is $r=r'$ and $\mathfrak{g}_i=\mathfrak{g'}_{i}$ for all $i\in\{1,\ldots,r\}$ up to permutation.
Indeed, if they hold, then we get $\sigma_i=\sigma'_i$ and $\tau_i=\tau'_i$ from $\sigma=\sigma'_{1}\times\cdots\times\sigma'_{r'}$, $\tau=\tau'_{1}\times\cdots\times\tau'_{r'}$ and \eqref{sigma,tau}. 
Hence we obtain $\Gamma_i=\Gamma'_{i}$.
Also, $g=g'_1\oplus\cdots\oplus g'_{r'}$ and $g_i=g|_{\mathfrak{g}_i\times\mathfrak{g}_i}$ imply $g_i=g'_i$.
Now, it is easy to see that
$(\mathfrak{g'}_{i'}\cap\mathfrak{g}_{i},\langle \sigma_{i}|_{\mathfrak{g'}_{i'}\cap\mathfrak{g}_{i}},\tau_{i}|_{\mathfrak{g'}_{i'}\cap\mathfrak{g}_{i}}\rangle)$ is a $\langle \sigma_{i}|_{\mathfrak{g'}_{i'}\cap\mathfrak{g}_{i}},\tau_{i}|_{\mathfrak{g'}_{i'}\cap\mathfrak{g}_{i}}\rangle$-symmetric Lie algebra for all $i'\in\{1,\ldots,r'\}$ and $i\in\{1,\ldots,r\}$.
Indeed, $\mathfrak{g'}_{i'}\cap\mathfrak{g}_{i}$ is an ideal of $\mathfrak{g}_{i}$.
Also, $\sigma_i(\mathfrak{g'}_{i'}\cap\mathfrak{g}_{i})=\sigma(\mathfrak{g'}_{i'}\cap\mathfrak{g}_{i})\subset\sigma(\mathfrak{g'}_{i'})\cap\sigma(\mathfrak{g}_{i})=\sigma'_{i'}(\mathfrak{g'}_{i'})\cap\sigma_i(\mathfrak{g}_{i})=\mathfrak{g'}_{i'}\cap\mathfrak{g}_{i}$ holds.
Similarly, $\tau_i(\mathfrak{g'}_{i'}\cap\mathfrak{g}_{i})\subset\mathfrak{g'}_{i'}\cap\mathfrak{g}_{i}$ holds.
From $\Gamma_i$-irreducibility of $(\mathfrak{g}_i,\Gamma_i)$, we have $\mathfrak{g'}_{i'}\cap\mathfrak{g}_{i}=\{0\}$ or $\mathfrak{g'}_{i'}\cap\mathfrak{g}_{i}=\mathfrak{g}_{i}$, implying $\mathfrak{g'}_{i'}\cap\mathfrak{g}_{i}=\{0\}$ or $\mathfrak{g}_{i}\subset\mathfrak{g'}_{i'}$.
By the same argument, it is also proved that $\mathfrak{g}_{i}\cap\mathfrak{g'}_{{i'}}=\{0\}$ or $\mathfrak{g'}_{{i'}}\subset\mathfrak{g}_{i}$ hold.
Hence, for each $i'_0\in\{1,\ldots,r'\}$ there exists an $i_0\in\{1,\ldots,r\}$ such that $\mathfrak{g'}_{i'_0}\subset\mathfrak{g}_{i_0}$.
Moreover, for this $i_0$, there exists an $i'_1\in\{1,\ldots,r'\}$ such that $\mathfrak{g}_{i_0}\subset\mathfrak{g'}_{i'_1}$.
Then we have $\mathfrak{g'}_{i'_0}\subset\mathfrak{g}_{i_0}\subset\mathfrak{g'}_{i'_1}$, and thus $\mathfrak{g'}_{i'_0}=\mathfrak{g'}_{i'_1}=\mathfrak{g}_{i_0}$ and $i'_0=i'_1$ hold.
This implies that for each $i'\in\{1,\ldots,r'\}$ there exists a unique $i\in\{1,\ldots,r\}$ such that $\mathfrak{g}_i=\mathfrak{g}'_{i'}$, and in particular, $r=r'$.
The proof is now complete.
\end{proof}
Note that $\Gamma_{i}$ in \eqref{decom._of_semisimple_Riem_Lie_alg.} can be isomorphic to not only $\mathbb{Z}_2\times\mathbb{Z}_2$ but also $\mathbb{Z}_2$. 
Indeed, $(\mathfrak{g}_1\oplus\mathfrak{g}_2,\langle\sigma\times\id_{\mathfrak{g}_2},\id_{\mathfrak{g}_1}\times\tau\rangle,g_1\oplus g_2)\cong(\mathfrak{g}_1,\langle\sigma\rangle,g_1)\oplus(\mathfrak{g}_2,\langle\tau\rangle,g_2)$ gives the example.
The examples of nontrivial type are given in subsection~\ref{subsection_The infinitesimal isotropy representation of the flag manifolds}.
The following example demonstrates that the assumption in theorem~\ref{decomposition of semisimple Lie algebra of nontrival type} is indispensable, that is, 
there can exist a non-perpendicular decomposition of  $(\mathfrak{g},\Gamma,g)$ if it is not of nontrivial type.
\begin{example}[An example of a non-perpendicular decomposition]\label{An example of a non-perpendicular decomposition}
Consider the flag manifold $SO(4)/S(O(1)\times O(1)\times O(1)\times O(1))$. 
Recall example~\ref{flag manifolds} and use the notations there.
We write $E_{ij}$ ($1 \le i \neq j \le 4$) for the skew-symmetric matrices whose $(i,j)$-entry is $1$, $(j,i)$-entry is $-1$, and all other entries are $0$ and 
define
\begin{align*}
&I_1 = \frac{1}{2} (E_{12} + E_{34}),\;\;
I_2 = \frac{1}{2} (E_{13} - E_{24}),\;\;
I_3 = -\frac{1}{2} (E_{14} + E_{23}) ,\\
&J_1 = \frac{1}{2} (E_{12} - E_{34}) ,\;\;
J_2 = \frac{1}{2} (E_{13} + E_{24}) ,\;\;
J_3 = \frac{1}{2} (E_{14} - E_{23}) .
\end{align*}
The simultaneous eigenspaces given by $\sigma$ and $\tau$ are
{\small
\begin{align*}
&\mathfrak{g}^{(+1,+1)}=\{0\}, \;\; 
\mathfrak{g}^{(+1,-1)}=\operatorname{span}\{I_1,J_1\}, \;\; \mathfrak{g}^{(-1,+1)}=\operatorname{span}\{I_2,J_2\}, \;\;
\mathfrak{g}^{(-1,-1)}=\operatorname{span}\{I_3,J_3\}.
\end{align*}}
On the other hand, the Lie bracket relations of $\mathfrak{g}$ are given by $[I_1,I_2]=I_3$, $[I_2,I_3]=I_1$, $[I_3,I_1]=I_2$, $[J_1,J_2]=J_3$, $[J_2,J_3]=J_1$, $[J_3,J_1]=J_2$ and $[I_k,J_l]=0$, $1\leq k,l\leq 3$.
Hence $\mathfrak{g}_I:=\operatorname{span}\{I_1,I_2,I_3\}$ and $\mathfrak{g}_J:=\operatorname{span}\{J_1,J_2,J_3\}$ give the decomposition of ideals of $\mathfrak{g}$.
Both $\mathfrak{g}_I$ and $\mathfrak{g}_J$ are isomorphic to $\mathfrak{so}(3)$ as a Lie algebra.
Thus we have decomposition
\begin{equation*}
(\mathfrak{so}(4),\langle\sigma,\tau\rangle)\cong(\mathfrak{g}_I,\langle\sigma_I,\tau_I\rangle)\oplus(\mathfrak{g}_J,\langle\sigma_J,\tau_J\rangle),
\end{equation*}
where $\sigma_I=\sigma|_{\mathfrak{g}_I}$, $\tau_I=\tau|_{\mathfrak{g}_I}$, $\sigma_J=\sigma|_{\mathfrak{g}_J}$ and $\tau_J=\tau|_{\mathfrak{g}_J}$.
Note that $(\mathfrak{g}_I,\langle\sigma_I,\tau_I\rangle)$ and $(\mathfrak{g}_J,\langle\sigma_J,\tau_J\rangle)$ are clearly $\mathbb{Z}_2\times\mathbb{Z}_2$-irreducible.
To simplify notation, we write $V_1$, $V_2$ and $V_3$ for $\mathfrak{g}^{(+1,-1)}$, $\mathfrak{g}^{(-1,+1)}$ and $\mathfrak{g}^{(-1,-1)}$, respectively.
Now, we define a symmetric bilinear form $g_k$ on $V_k$ by $g_k(I_k,I_k)=\kappa_k$, $g_k(J_k,J_k)=\lambda_k$, $g_k(I_k,J_k)=c_k$, $\kappa_k,\lambda_k,c_k\in\mathbb{R}$ and $\kappa_k,\lambda_k>0$, for each $k=1,2,3$.
\begin{lemma}
$g_k$ is positive definite if and only if $-\sqrt{\kappa_k\lambda_k}<c_k<\sqrt{\kappa_k\lambda_k}$.
\end{lemma}
\begin{proof}
The matrix representation $A_{g_k}$ of $g_k$ with respect to the basis $\{I_k, J_k\}$ is given by
$\begin{pmatrix}
\kappa_k & c_k  \\
c_k & \lambda_k
\end{pmatrix}$.
Since 
$\tr A_{g_k}=\kappa_k+\lambda_k>0$, the bilinear form $g_k$ is positive definite if and only if 
$\det A_{g_k}
>0$.
Thus we obtain the conclusion.
\end{proof}
Let us assume that $-\sqrt{\kappa_k\lambda_k}<c_k<\sqrt{\kappa_k\lambda_k}$ for each $k=1,2,3$.
We define an inner product $g$ on $\mathfrak{m}=V_1\oplus V_2\oplus V_3$ by $g=g_1\oplus g_2\oplus g_3$ (orthogonal direct sum).
Then $(\mathfrak{g},\langle\sigma,\tau\rangle,g)$ is a Riemannian $\mathbb{Z}_2\times\mathbb{Z}_2$-symmetric Lie algebra.
Let us denote $V_{I_k}:=\operatorname{span}\{I_k\}$ and $V_{J_k}:=\operatorname{span}\{J_k\}$ for each $k$.
Then $(\mathfrak{g}_I,\langle\sigma_I,\tau_I\rangle,g|_{\mathfrak{g}_I\times\mathfrak{g}_I})$ and $(\mathfrak{g}_J,\langle\sigma_J,\tau_J\rangle,g|_{\mathfrak{g}_J\times\mathfrak{g}_J})$ are  Riemannian $\mathbb{Z}_2\times\mathbb{Z}_2$-symmetric Lie algebras.
However, the decomposition $V_k=V_{I_k}\oplus V_{J_k}$ is not orthogonal with respect to $g_k$ if $c_k\neq0$.
Thus if $c_k\neq0$, we obtain
\begin{equation*}
(\mathfrak{so}(4),\langle\sigma,\tau\rangle,g)\ncong(\mathfrak{so}(3),\langle\sigma_I,\tau_I\rangle,g|_{\mathfrak{g}_I\times\mathfrak{g}_I})\oplus(\mathfrak{so}(3),\langle\sigma_J,\tau_J\rangle,g|_{\mathfrak{g}_J\times\mathfrak{g}_J}).
\end{equation*}
\end{example}
\subsection{Semisimple irreducible orthogonal $\Gamma$-symmetric Lie algebras}
Irreducible orthogonal $\Gamma$-symmetric Lie algebras with $\Gamma\cong\mathbb{Z}_2$ or $\mathbb{Z}_2\times\mathbb{Z}_2$ are classified by the following lemmas. 
Note that they are effective.
\renewcommand{\theenumi}{\arabic{enumi}}
\begin{lemma}[cf. \cite{helgason1979differential}]\label{types of irr.-O.S.L.}
    A semisimple irreducible orthogonal symmetric Lie algebra $(\mathfrak{g},\sigma)$ is isomorphic to one of the followings: 
    \begin{enumerate}
    \item $\mathfrak{g}$ is simple. \label{1st type of O.S.L.}
    \item $(\mathfrak{g}_0\oplus\mathfrak{g}_0,\theta)$. Here, $\mathfrak{g}_0$ is a compact simple Lie algebra and an involution $\theta$ denotes
\begin{equation*}
\theta(X,Y)=(Y,X) \;\;\;(X,Y\in\mathfrak{g}_0).
\end{equation*}\label{2nd type of O.S.L.}
    \end{enumerate}
\end{lemma}
\vspace{-20pt}
\begin{lemma}[cf. \cite{ikawa2025intersectionrealflagmanifolds},  \cite{Matsuki}]\label{types of irr.-Z2Z2}
     A semisimple irreducible orthogonal $\mathbb{Z}_2\times\mathbb{Z}_2$-symmetric Lie algebra $(\mathfrak{g},\mathbb{Z}_2\times\mathbb{Z}_2)$ is isomorphic to one of the followings: 
     \begin{enumerate}
\item \label{simples-types of irr.-Z2Z2}$\mathfrak{g}$ is simple.
\item \label{two sum of simples-types of irr.-Z2Z2}
    $(\mathfrak{g}_0\oplus\mathfrak{g}_0,\theta,\eta)$. Here, $\mathfrak{g}_0$ is a simple Lie algebra and two involutions $\theta$, $\eta$ of $\mathfrak{g}_0\oplus\mathfrak{g}_0$ are given by
    \begin{equation*}
        \theta(X,Y)=(\theta_0(X),\theta_0(Y)),\;\eta(X,Y)=(Y,X) \;\;(X,Y\in\mathfrak{g}_0),
        \end{equation*}
    where the pair $(\mathfrak{g}_0,\theta_0)$ is an orthogonal symmetric Lie algebra.
\item \label{four sum of simples-types of irr.-Z2Z2}
$(\mathfrak{g}_0\oplus\mathfrak{g}_0\oplus\mathfrak{g}_0\oplus\mathfrak{g}_0,\theta,\eta)$. Here, $\mathfrak{g}_0$ is a compact simple Lie algebra and two involutions $\theta$, $\eta$ of $\mathfrak{g}_0\oplus\mathfrak{g}_0\oplus\mathfrak{g}_0\oplus\mathfrak{g}_0$ are given by
\begin{equation*}
\theta(X,Y,Z,W)=(Y,X,W,Z), \;
\eta(X,Y,Z,W)=(Z,W,X,Y)
\;\;(X,Y,Z,W\in\mathfrak{g}_0).
\end{equation*}
     \end{enumerate}
\end{lemma}
Simple orthogonal symmetric Lie algebras have been classified (cf. \cite{helgason1979differential}), and 
simple $\mathbb{Z}_2\times\mathbb{Z}_2$-symmetric Lie algebras have also been done in \cite{Bahturin&Goze} and \cite{Kollross}.
Therefore, any semisimple $\Gamma$-irreducible orthogonal $\Gamma$-symmetric Lie algebra can be determined by these results.
Furthermore, by investigating the infinitesimal isotropy representation for each of them, we can identify those of nontrivial type.
In addition, the irreducible decomposition of the infinitesimal isotropy representation determines the inner products admitted on its standard complement.
By this procedure, together with theorem~\ref{decomposition of semisimple Lie algebra of nontrival type}, the classification of effective semisimple Riemannian $\mathbb{Z}_2\times\mathbb{Z}_2$-symmetric Lie algebras of nontrivial type can be completed.

As part of this procedure, we next consider the infinitesimal isotropy representations of the cases~\ref{two sum of simples-types of irr.-Z2Z2} and \ref{four sum of simples-types of irr.-Z2Z2} in lemma~\ref{types of irr.-Z2Z2}.
It is well-known that the representations of semisimple irreducible orthogonal symmetric Lie algebras are nontrivial and irreducible (cf. \cite{helgason1979differential}).
In the case~\ref{two sum of simples-types of irr.-Z2Z2} in lemma~\ref{types of irr.-Z2Z2}, the representations on simultaneous eigenspaces $\mathfrak{g}^{(+1,-1)}=\{(X,-X)\in\mathfrak{g}_0\oplus\mathfrak{g}_0\mid X\in\mathfrak{g}_0^{\theta_0}\}$, $\mathfrak{g}^{(-1,+1)}=\{(X,X)\in\mathfrak{g}_0\oplus\mathfrak{g}_0\mid X\in\mathfrak{g}_0^{-\theta_0}\}$ and $\mathfrak{g}^{(-1,-1)}=\{(X,-X)\in\mathfrak{g}_0\oplus\mathfrak{g}_0\mid X\in\mathfrak{g}_0^{-\theta_0}\}$ are equivalent to the adjoint representation of $\mathfrak{g}_0^{\theta_0}$ on $\mathfrak{g}_0^{\theta_0}$ or that on $\mathfrak{g}_0^{-\theta_0}$.
Similarly, in the case~\ref{four sum of simples-types of irr.-Z2Z2} in lemma~\ref{types of irr.-Z2Z2}, the representations on simultaneous eigenspaces are equivalent to the adjoint representation of $\mathfrak{g}_0$ on $\mathfrak{g}_0$.
Hence we obtain the next proposition:
\begin{proposition}\label{the infinitesimal isotropy representation about irreducible Lie algebras}
In the cases~\ref{two sum of simples-types of irr.-Z2Z2} and \ref{four sum of simples-types of irr.-Z2Z2} in lemma~\ref{types of irr.-Z2Z2}, the infinitesimal isotropy representation satisfies the following properties:
\begin{itemize}
    \item In the case~\ref{two sum of simples-types of irr.-Z2Z2} with $\dim \mathfrak{g}_0^{\theta_0}=1$,  the representation on $\mathfrak{g}^{(+1,-1)}$ is trivial. 
    Also, the representations on $\mathfrak{g}^{(-1,+1)}$ and $\mathfrak{g}^{(-1,-1)}$ are nontrivial and irreducible. 
    \item In the case~\ref{two sum of simples-types of irr.-Z2Z2} with $\dim \mathfrak{g}_0^{\theta_0}\geq2$, the representation on each simultaneous eigenspace is nontrivial and irreducible. 
    \item In the cases~\ref{four sum of simples-types of irr.-Z2Z2}, the representation on each simultaneous eigenspace is nontrivial and irreducible.
\end{itemize}    
\end{proposition}
The latter two cases in proposition~\ref{the infinitesimal isotropy representation about irreducible Lie algebras} are of nontrivial type.
\subsection{Examples of nontrivial type}\label{subsection_The infinitesimal isotropy representation of the flag manifolds}
In this subsection, we give examples of nontrivial type.
We mention the generalized Wallach spaces, observe the flag manifolds and compute the infinitesimal isotropy representations.
\begin{definition}[cf. \cite{Nikonorov}]
A Riemannian $\mathbb{Z}_2\times\mathbb{Z}_2$-symmetric space $(G/K,\Gamma,g)$ is called a generalized Wallach space if $G$ is compact semisimple and if the infinitesimal isotropy representation on each simultaneous eigenspace is irreducible.
\end{definition}
The generalized Wallach spaces are classified in \cite{Nikonorov}. 
They give a lot of examples of nontrivial type.
For instance, the flag manifold $SO(l)/S(O(r_1)\times O(r_2)\times O(r_3)\times O(r_4))$ with $r_1r_2r_3\neq0$ and with $r_4=0$ has three $\mathfrak{k}$-subspaces $\mathfrak{g}^b$, $\mathfrak{g}^c$ and $\mathfrak{g}^d$ on which its representations act nontrivially and irreducibly except for the case $(r_1,r_2,r_3,r_4)=(1,1,1,0)$ or $(2,1,1,0)$ (see \cite[Table 1]{Nikonorov} and example~\ref{The flag manifold with (r_1,r_2,r_3,r_4)=(2,1,1,0)}).
The flag manifolds give examples from the viewpoint of not only the product and decomposition of Riemannian $\Gamma$-symmetric spaces but also nontrivial type.

\begin{example}[The flag manifold with $(r_1,r_2,r_3,r_4)=(2,1,1,0)$]
\label{The flag manifold with (r_1,r_2,r_3,r_4)=(2,1,1,0)}
The $\mathbb{Z}_2\times\mathbb{Z}_2$-symmetric Lie algebra $(\mathfrak{so}(4),\langle\sigma,\tau\rangle)$ given in example~\ref{flag manifolds} is isomorphic to the case~\ref{two sum of simples-types of irr.-Z2Z2} in lemma~\ref{types of irr.-Z2Z2}.
Let us check it.
Let $I_1$, $I_2$, $I_3$, $J_1$, $J_2$ and $J_3$ be matrices defined in example~\ref{An example of a non-perpendicular decomposition}.
Then the simultaneous eigenspaces given by $\sigma$ and $\tau$ are $\mathfrak{k}=\langle I_1+J_1\rangle$, $\mathfrak{g}^{(-1,+1)}=\langle I_2+J_2,I_3+J_3\rangle$, $\mathfrak{g}^{(+1,-1)}=\langle I_2-J_2,I_3-J_3\rangle$ and $\mathfrak{g}^{(-1,-1)}=\langle I_1-J_1\rangle$.
On the other hand, we define $\mathfrak{g}_0=\langle I_1,I_2,I_3\rangle$ and also define $\theta_0\in\Aut(\mathfrak{g}_0)$ by $\theta_0(I_1)=I_1$, $\theta_0(I_2)=-I_2$ and $\theta_0(I_3)=-I_3$.
Then the irreducible $\mathbb{Z}_2\times\mathbb{Z}_2$-symmetric Lie algebra $(\mathfrak{g}_0\oplus\mathfrak{g}_0,\theta,\eta)$ defined by the case~\ref{two sum of simples-types of irr.-Z2Z2} in lemma~\ref{types of irr.-Z2Z2} is obtained by $\theta=\theta_0\times\theta_0$ and the transposition map $\eta$. 
A diffeomorphism $\varphi:\mathfrak{so}(4)\to\mathfrak{g}_0\oplus\mathfrak{g}_0$ and a group isomorphism $\Phi:\langle\sigma,\tau\rangle\to\langle\theta,\eta\rangle$ defined by
\begin{align*}
&\varphi(I_i)=(I_i,0),\quad\varphi(J_i)=(0,I_i)\qquad(i=1,2,3), \\
&\Phi(\tau)=\eta,\quad \Phi(\sigma\tau)=\theta
\end{align*}
give an isomorphism $(\varphi,\Phi)$ from $(\mathfrak{so}(4),\langle\sigma,\tau\rangle)$ to $(\mathfrak{g}_0\oplus\mathfrak{g}_0,\langle\theta,\eta\rangle)$.
Therefore, since $\dim\mathfrak{g}_0^{\theta_0}=\dim\langle I_1\rangle=1$ and proposition~\ref{the infinitesimal isotropy representation about irreducible Lie algebras}, the infinitesimal isotropy representation on $\mathfrak{g}^{(+1,-1)}$ (resp. on $\mathfrak{g}^{(-1,+1)}$ and $\mathfrak{g}^{(-1,-1)}$) is trivial (resp. nontrivial and irreducible). 
\end{example}
According to \cite{Arvanitoyeorgos-Sakane-Statha}, except for the case of $r_1=r_2=r_3=r_4=1$, the flag manifolds have six irreducible $\mathfrak{k}$-subspaces if $r_4\neq0$, which are expressed as 
\begin{align*}
&\left(
\begin{array}{cccc}
0 & A_1 & 0 & 0 \\
-{}^{t}\!A_1 & 0 & 0 & 0 \\
0 & 0 & 0 & 0 \\
0 & 0 & 0 & 0
\end{array}
\right),
\left(
\begin{array}{cccc}
0 & 0 & 0 & 0 \\
0 & 0 & 0 & 0 \\
0 & 0 & 0 & A_2 \\
0 & 0 & -{}^{t}\!A_2 & 0
\end{array}
\right),
\left(
\begin{array}{cccc}
0 & 0 & B_1 & 0 \\
0 & 0 & 0 & 0 \\
-{}^{t}\!B_1 & 0 & 0 & 0 \\
0 & 0 & 0 & 0
\end{array}
\right), \\
&\left(
\begin{array}{cccc}
0 & 0 & 0 & 0 \\
0 & 0 & 0 & B_2 \\
0 & 0 & 0 & 0 \\
0 & -{}^{t}\!B_2 & 0 & 0
\end{array}
\right),
\left(
\begin{array}{cccc}
0 & 0 & 0 & C_1 \\
0 & 0 & 0 & 0 \\
0 & 0 & 0 & 0 \\
-{}^{t}C_1 & 0 & 0 & 0
\end{array}
\right),
\left(
\begin{array}{cccc}
0 & 0 & 0 & 0 \\
0 & 0 & C_2 & 0 \\
0 & -{}^{t}C_2 & 0 & 0 \\
0 & 0 & 0 & 0
\end{array}
\right)
\end{align*}
respectively.
Combining the above, we get the following proposition.
It is divided into cases depending on whether the isotropy subgroup is discrete, whether $l=4$ and whether $r_4=0$ or not.
\begin{proposition}\label{Whether isotropy rep. of the flag mfd is triv. or not.}
Let $SO(l)/S(O(r_1)\times O(r_2)\times O(r_3)\times O(r_4))$ be the flag manifold with $r_1r_2r_3\neq0$.
Then the following properties for the infinitesimal isotropy representation hold:
    \renewcommand{\theenumi}{(\arabic{enumi})}
\begin{enumerate}
    \item If $(r_1,r_2,r_3,r_4)=(1,1,1,0)$, then the infinitesimal isotropy representation on each $\mathfrak{g}^a$ ($a\in\Gamma^{\times}$) is trivial.
    \item If $(r_1,r_2,r_3,r_4)=(2,1,1,0)$, then the infinitesimal isotropy representation on each $\mathfrak{g}^{(+1,-1)}$ and $\mathfrak{g}^{(-1,+1)}$ (resp. on $\mathfrak{g}^{(-1,-1)}$) is nontrivial and irreducible (resp. is trivial).  
    \item \label{isotropy rep. w.r.t. SO(5)/S(O(2)times O(2)times O(1))} 
    If $r_1,r_2\geq2$, $r_3\geq1$ and $r_4=0$, then the infinitesimal isotropy representation on each $\mathfrak{g}^a$ ($a\in\Gamma^{\times}$) is nontrivial and irreducible.
    \item If $(r_1,r_2,r_3,r_4)=(1,1,1,1)$, then the infinitesimal isotropy representation on each $\mathfrak{g}^a$ ($a\in\Gamma^{\times}$) is trivial.
    \item \label{isotropy rep. w.r.t. SO(5)/S(O(2)times O(1)times O(1)times O(1))} 
    If $r_1\geq2$ and $r_2,r_3,r_4\geq1$, then the infinitesimal isotropy representation on each $\mathfrak{g}^a$ ($a\in\Gamma^{\times}$) has two isotropy summands.
\end{enumerate}
\end{proposition}
In the case~\ref{isotropy rep. w.r.t. SO(5)/S(O(2)times O(1)times O(1)times O(1))}, consider the situation where $r_i=r_j=1$ ($i,j\in\{2,3,4\}$, $i\neq j$).
For example, if $r_2=r_3=1$, then the infinitesimal isotropy representation on the irreducible $\mathfrak{k}$-subspace expressed by $C_2\in\mathcal{M}(r_2,r_3)$ is trivial since $\mathfrak{so}(r_2)=\mathfrak{so}(r_3)=\{0\}$.
Hence, the flag manifolds in the case~\ref{isotropy rep. w.r.t. SO(5)/S(O(2)times O(1)times O(1)times O(1))} such that at most one of $r_2$, $r_3$, $r_4$ is equal to $1$ provide examples that are not generalized Wallach spaces, but are of nontrivial type.
\section{Relationship between Riemannian $\Gamma$-symmetric triples, Riemannian $\Gamma$-symmetric Lie algebras and Riemannian $\Gamma$-symmetric spaces}\label{sec. the relationship between Riemannian Gamma-symmetric triples, Riemannian Gamma-symmetric Lie algebras and Riemannian Gamma-symmetric spaces}

\subsection{Relationship between Riemannian $\Gamma$-symmetric spaces and Riemannian $\Gamma$-symmetric triples}\label{The relationship between Riemannian Gamma-symmetric spaces and Riemannian Gamma-symmetric triples}
Let $\Gamma$ be a finite abelian group (not necessarily $\Gamma\cong\mathbb{Z}_2$ or $\mathbb{Z}_2\times\mathbb{Z}_2$) and $(M,\Gamma,g)$ a Riemannian $\Gamma$-symmetric space again. 
Define $G:=\Aut_0(M,\Gamma,g)$ and $K:=\{a\in G \mid a(o)=o\}$, where $o$ denotes the origin of $M$. 
Then we obtain the identification of $M$ and $G/K$ as a smooth manifold. 
For each $\gamma\in\Gamma$, we define $\sigma_\gamma\in\Aut(G)$ by
\begin{equation*}
    \sigma_\gamma:G\to G; a\mapsto \gamma_o a\gamma_o^{-1}.
\end{equation*}
Note that $\gamma_o a\gamma_o^{-1}\in G$ since $\gamma_o\in\Aut(M,\Gamma,g)$ and $G$ is a normal subgroup of $\Aut(M,\Gamma,g)$. 
Hence the map $\Gamma\to\Aut(G);\gamma\mapsto\sigma_{\gamma}$ is an injective homomorphism. 
We shall prove $F_0(\Gamma,G)\subset K\subset F(\Gamma,G)$, where $\Gamma$ also denotes $\{\sigma_{\gamma}\mid\gamma\in\Gamma\}$. 
For each $k\in K$, we obtain $\sigma_{\gamma}(k)=k$ for all $\gamma\in\Gamma$ since $k\circ\gamma_o=\gamma_{k(o)}\circ k=\gamma_{o}\circ k$. 
Hence $K\subset F(\Gamma,G)$. 
Next, if $\exp(tX)\in F_0(\Gamma,G)$ for all $t\in \mathbb{R}$ and for all $X\in \mathfrak{g}$, then $\exp(tX)(o)\in \mathrm{Fix}(\Gamma_o,M)$. 
Since $o\in M$ is an isolated point of $F(\Gamma_o,M)$, we obtain $\exp(tX)(o)=o$, that is, $\exp(tX)\in K$ for sufficiently small $t$.
Since $K$ is a subgroup of $G$, $\exp(tX)\in K$ for all $t\in\mathbb{R}$. 
This means that $F_0(\Gamma,G)\subset K$.  
Hence $(G,K,\Gamma)$ is a $\Gamma$-symmetric triple. 
Also, since $K$ is compact, $(G,K,\Gamma)$ satisfies the condition (i).

Assume that $\Gamma\cong\mathbb{Z}_2$ or $\mathbb{Z}_2\times\mathbb{Z}_2$.
Recall that $\mathfrak{m}$ denotes the standard complement and $\pi_*(X)=X^{\sharp}_o$ for $X\in\mathfrak{m}$.
The identification of $\mathfrak{m}$ and $T_oM$ via $\pi_*$ induces an inner product $g$ on $\mathfrak{m}$ from $g_o$.
Hence $g$ satisfies \ref{ga is Ad(K)-inv.}.
Since the differential of $\sigma_{\gamma}$, also denoted by $\sigma_{\gamma}$, satisfies $(\gamma_o)_*\circ\pi_*=\pi_*\circ(\sigma_{\gamma}|_{\mathfrak{m}})$, then the map $\sigma_{\gamma}|_{\mathfrak{m}}$ preserves $g$, which implies \ref{deco. is perp w.r.t. g}.
Thus $(G,K,\Gamma,g)$ is a Riemannian $\Gamma$-symmetric triple.
When we construct a Riemannian $\Gamma$-symmetric space from $(G,K,\Gamma,g)$ by the method in section~\ref{sec. Gamma-symmetric triples}, it is isomorphic to the original Riemannian $\Gamma$-symmetric space $(M,\Gamma,g)$. 

Let $(G',K',\Gamma',g')$ be a Riemannian $\Gamma'$-symmetric triple constructed from a Riemannian $\Gamma'$-symmetric space $(M',\Gamma',g')$. 
If $(M,\Gamma,g)$ and $(M',\Gamma',g')$ are isomorphic, then $(G,K,\Gamma,g)\cong(G',K',\Gamma',g')$. 
This is proved as follows:
Let $(\varphi,\Phi)$ be an isomorphism from $(M,\Gamma,g)$ to $(M',\Gamma',g')$. 
We can assume that $\varphi(o)=o'$, where $o'$ denotes the origin of $(M',\Gamma',g')$. 
Define a Lie group isomorphism $\hat{\varphi}:G\to G'$ by 
\begin{equation*}
    \hat{\varphi}(a):=\varphi\circ a\circ\varphi^{-1}\;\;(a\in G).
\end{equation*}
Take $\gamma\in\Gamma$. For $\sigma_{\gamma}$, we define $\hat{\Phi}(\sigma_{\gamma})\in \Aut(G')$ by
\begin{equation*}
    \hat{\Phi}(\sigma_{\gamma})(a'):=\Phi(\gamma)_{o'}\circ a'\circ\Phi(\gamma)_{o'}^{-1}\;\;(a'\in G').
\end{equation*}
Then $(\hat{\varphi},\hat{\Phi})$ satisfies the condition (A) since $(\varphi,\Phi)$ is an isomorphism.
Furthermore, we obtain (B) and (C) since $\varphi(o)=o'$ and $\varphi$ is an isometry, respectively. 
Thus $(\hat{\varphi},\hat{\Phi})$ gives an isomorphism between $(G,K,\Gamma,g)$ and $(G',K',\Gamma',g')$.

\vspace{0.2cm}
 A $\Gamma$-symmetric triple $(G,K,\Gamma)$ is said to be \textit{effective} (resp. \textit{almost effective}) if there exists no nontrivial (resp. nondiscrete) normal subgroup of $G$ contained in $K$. 
 Also, Riemannian $\Gamma$-symmetric triple $(G,K,\Gamma,g)$ is said to be \textit{effective} (resp. \textit{almost effective}) if $(G,K,\Gamma)$ is \textit{effective} (resp. \textit{almost effective}).
\begin{lemma}\label{exists effective one}
    For each Rimannian $\Gamma$-symmetric space $(M,\Gamma,g)$, there exists an effective Riemannian $\Gamma$-symmetric triple $(\bar{G},\bar{K},\Gamma,g)$ such that $(M,\Gamma,g)\cong(\bar{G}/\bar{K},\Gamma,g)$.
\end{lemma}
\begin{proof}
    We refer to the main idea of the proof. 
    Let $(G,K,\Gamma,g)$ be the Rimannian $\Gamma$-symmetric triple constructed from $(M,\Gamma,g)$ by the method mentioned in subsection~\ref{Riemannian Gamma-symmetric triples}.     
    Let $A$ be the maximal normal subgroup of $G$ contained in $K$ such that its canonical action on $G/K$ is trivial. 
    We define quotient Lie groups $\bar{G}$ and $\bar{K}$ by $\bar{G}=G/A$ and $\bar{K}=K/A$. 
    We write $\bar{\gamma}$ as an automorphism of $\bar{G}$ induced by $\gamma\in\Gamma$, and use the same letters $\Gamma$ for $\{\bar{\gamma}\mid \gamma\in\Gamma\}$. 
    Then $(\bar{G},\bar{K},\Gamma,g)$ is the effective Riemannian $\Gamma$-symmetric triple such that the standard complement coincides with that of $(G,K,\Gamma,g)$ and $(M,\Gamma,g)\cong(\bar{G}/\bar{K},\Gamma,g)$.
\end{proof}

\begin{lemma}
    If $(M,\Gamma,g)$ and $(M',\Gamma',g')$ are locally isomorphic, then there exist $(G,K,\Gamma,g)$ and $(G',K',\Gamma',g')$ such that they are effective and isomorphic, $(M,\Gamma,g)\simeq(G/K,\Gamma,g)$ and $(M',\Gamma',g')\simeq (G'/K',\Gamma',g')$.
\end{lemma}
\begin{proof}
    Let $(\tilde{M},\Gamma,\tilde{g})$ be Riemannian universal covering $\Gamma$-symmetric spaces of $(M,\Gamma,g)$ and $(\tilde{M'},\Gamma',\tilde{g'})$ that of $(M',\Gamma',g')$. By lemma~\ref{exists effective one}, we obtain $(\tilde{G},\tilde{K},\Gamma,\tilde{g})$ and $(\tilde{G'},\tilde{K'},\Gamma',\tilde{g'})$ such that they are effective and $(M,\Gamma,g)\simeq (\tilde{G}/\tilde{K},\Gamma,\tilde{g})$ and $(M',\Gamma',g')\simeq (\tilde{G'}/\tilde{K'},\Gamma',\tilde{g'})$. 
    Since $(\tilde{M},\Gamma,\tilde{g})$ and $(\tilde{M'},\Gamma',\tilde{g'})$ are isomorphic, $(\tilde{G},\tilde{K},\Gamma,\tilde{g})$ and $(\tilde{G'},\tilde{K'},\Gamma',\tilde{g'})$ are isomorphic. 
\end{proof}
\subsection{Relationship between Riemannian $\Gamma$-symmetric triples and Riemannian $\Gamma$-symmetric Lie algebras}\label{The relationship between Riemannian Gamma-symmetric triples and Riemannian Gamma-symmetric Lie algebras}
Supposse that $\Gamma\cong\mathbb{Z}_2$ or $\mathbb{Z}_2\times\mathbb{Z}_2$.
Let $(G,K,\Gamma,g)$ be a Riemannian $\Gamma$-symmetric triple and  $(\mathfrak{g},\Gamma)$ the $\Gamma$-symmetric Lie algebra associated with $(G,K,\Gamma)$. 
Then $(\mathfrak{g},\Gamma,g)$ is the Riemannian $\Gamma$-symmetric Lie algebra and this $(\mathfrak{g},\Gamma,g)$ is called a Riemannian $\Gamma$-symmetric Lie algebra associated with $(G,K,\Gamma,g)$. 
This construction is commutative with each isomorphic relation. 
In addition, almost effective $\Gamma$-symmetric triples correspond to effective $\Gamma$-symmetric Lie algebras. 
Furthermore, this construction is commutative with the direct product of Riemannian $\Gamma$-symmetric triples and the direct sum of Riemannian $\Gamma$-symmetric Lie algebras. 

Let $(\mathfrak{g},\Gamma,g)$ be a Riemannian $\Gamma$-symmetric Lie algebra. 
Let $G$ be a unique simply connected Lie group such that $\mathfrak{g}=\Lie G$. 
For each $\gamma\in\Aut(\mathfrak{g})$, there exists an automorphism of $G$, also denoted by $\gamma$, such that its differential coincides with the original $\gamma$. 
Define $K:=F_0(\Gamma,G)$ and then $(G,K,\Gamma,g)$ is the Riemannian $\Gamma$-symmetric triple. 
Since $K$ is connected and $G$ is simply connected, we can construct the simply connected Riemannian $\Gamma$-symmetric space $(M,\Gamma,g)$ from $(G,K,\Gamma,g)$ by the method mentioned in subsection~\ref{Riemannian Gamma-symmetric triples}. 
This construction is commutative with each isomorphic relation. 
Furthermore, this construction is commutative with the direct sum of Riemannian $\Gamma$-symmetric Lie algebras and the direct product of Rimannian $\Gamma$-symmetric spaces. 
\subsection{Correspondence between Riemannian $\Gamma$-symmetric Lie algebras and Rimannian $\Gamma$-symmetric spaces}
The aim of this subsection is to show that the decomposition of $\mathbb{Z}_2\times\mathbb{Z}_2$-symmetric Lie algebras given in section~\ref{sec. decomposition of nontrivial type} corresponds to the decomposition of Riemannian $\mathbb{Z}_2\times\mathbb{Z}_2$-symmetric spaces.
\begin{proposition}\label{one-to-one correspondence}
    Assume that $\Gamma\cong\mathbb{Z}_2$ or $\mathbb{Z}_2\times\mathbb{Z}_2$. 
    There exists a one-to-one correspondence between the equivalence class $\mathscr{G}$ of effective Riemannian $\Gamma$-symmetric Lie algebras with respect to $\cong$ and the equivalence class $\mathscr{M}$ of Rimannian $\Gamma$-symmetric spaces with respect to $\simeq$.
\end{proposition}
\begin{proof}
As already mentioned in subsection~\ref{The relationship between Riemannian Gamma-symmetric triples and Riemannian Gamma-symmetric Lie algebras}, we can construct an equivalence class of simply connected Riemannian $\Gamma$-symmetric spaces from $\mathscr{G}$. 
This construction induces the map from $\mathscr{G}$ to $\mathscr{M}$. 
Conversely, from subsection~\ref{The relationship between Riemannian Gamma-symmetric spaces and Riemannian Gamma-symmetric triples} and subsection~\ref{The relationship between Riemannian Gamma-symmetric triples and Riemannian Gamma-symmetric Lie algebras}, we can construct effective Riemannian $\Gamma$-symmetric Lie algebras from Rimannian $\Gamma$-symmetric spaces. 
This construction induces the map from $\mathscr{M}$ to $\mathscr{G}$. 
We check that these maps are mutually inverse. 
Let $(M,\Gamma,g)$ be a Rimannian $\Gamma$-symmetric space. 
Then we obtain an effective Riemannian $\Gamma$-symmetric Lie algebra $(\mathfrak{g},\Gamma,g)$ from $(M,\Gamma,g)$. 
For this $(\mathfrak{g},\Gamma,g)$, let $(\tilde{M},\Gamma,\tilde{g})$ be the simply connected Riemannian $\Gamma$-symmetric space constructed from $(\mathfrak{g},\Gamma,g)$. 
From the construction, $(\tilde{M},\Gamma,\tilde{g})$ is the universal covering Riemannian $\Gamma$-symmetric space of $(M,\Gamma,g)$. 
Hence we have $(M,\Gamma,g)\simeq (\tilde{M},\Gamma,\tilde{g})$.
\end{proof}

Since the correspondence in proposition~\ref{one-to-one correspondence} is commutative with direct sum and direct product relation, we have the following corollary:
\begin{corollary}\label{correspomdemce between Gamma-irr. and locally irr.}
Assume that $\Gamma\cong\mathbb{Z}_2$ or $\mathbb{Z}_2\times\mathbb{Z}_2$. 
Let $(\mathfrak{g},\Gamma,g)$ be a semisimple Riemannian $\Gamma$-symmetric Lie algebra of nontrivial type and $(M,\Gamma,g)$ a Riemannian $\Gamma$-symmetric space corresponding to $(\mathfrak{g},\Gamma,g)$ in the sense of proposition~\ref{one-to-one correspondence}. 
Then $(M,\Gamma,g)$ is locally irreducible if and only if $(\mathfrak{g},\Gamma,g)$ is $\Gamma$-irreducible.
\end{corollary}
A Riemannian $\Gamma$-symmetric space $(M,\Gamma,g)$ is called \textit{semisimple} if the Lie group $\Aut_0(M,\Gamma,g)$ is semisimple.
Also, $(M,\Gamma,g)$ is a semisimple Riemannian $\Gamma$-symmetric space \textit{of nontrivial type} if the corresponding effective semisimple Riemannian $\Gamma$-symmetric Lie algebra is of nontrivial type.
Noting that ordinary semisimple Riemannian symmetric spaces are always of nontrivial type, corollary~\ref{correspomdemce between Gamma-irr. and locally irr.} and theorem~\ref{decomposition of semisimple Lie algebra of nontrival type} imply the following theorem, which 
leads to the local classification of semisimple Riemannian $\mathbb{Z}_2\times\mathbb{Z}_2$-symmetric spaces of nontrivial type.
\begin{theorem}\label{decomposition of Gamma-symmetric spaces}
Every semisimple Riemannian $\mathbb{Z}_2\times\mathbb{Z}_2$-symmetric space of nontrivial type is locally isomorphic to the direct product of irreducible semisimple Riemannian $\mathbb{Z}_2\times\mathbb{Z}_2$-symmetric spaces of nontrivial type or irreducible semisimple Riemannian symmetric spaces. 
Furthermore, the decomposition is unique up to the order.
\end{theorem}
\bibliography{bibtex}
\bibliographystyle{plain}
\nocite{Equivariant_formality_of_the_isotropy_action_on_Z2Z2-symmetric_spaces}
\end{document}